\documentclass[11pt]{article}
\usepackage{mathrsfs}
\usepackage{amssymb}
\usepackage{amsmath}
\usepackage{amsbsy}
\usepackage{epsfig}
\usepackage{enumerate}
\usepackage{bm}
\usepackage{bbm}
 \usepackage{color}

\newcommand{\red}[1]{\textcolor{red}{#1}}

\usepackage[colorlinks,linkcolor=blue]{hyperref}

\newcommand{\MR}[1]{\href{http://www.ams.org/mathscinet-getitem?mr=#1}{\textcolor{blue}{MR-#1}}}%

\def\proclaim#1{\par \smallskip\noindent {\bf #1}\bgroup\it\ }
\def\endproclaim{\egroup\par\smallskip}

\newtheorem{lemma}{Lemma}[section]
\newtheorem{theorem}{Theorem}[section]
\newtheorem{proposition}{Proposition}[section]

\newtheorem{corollary}{Corollary}[section]

\newtheorem{remark}{Remark}[section]

\newbox\TempBox \newbox\TempBoxA

\def\pr{\textsf{P}} 
\def\ep{\textsf{E}} 

\def\Var{\textsf{Var}} 

\def\underwiggle 1{
\ifmmode\setbox\TempBox=\hbox{$ 1$}\else\setbox\TempBox=\hbox{
1}\fi \setbox\TempBoxA=\hbox to \wd\TempBox{\hss\char'176\hss}
\rlap{\copy\TempBox}\smash{\lower9pt\hbox{\copy\TempBoxA}} }

\begin{document}

\title{\bf Chung-type laws of the iterated logarithm for  $m$-fold weighted integrated fractional processes\thanks{Supported by grants from National Key R\&D Program of China (Grant No. 2024YFA1013502),  NSF of China  (Grant No. U23A2064) and   the Summit Advancement Disciplines of Zhejiang Province (Zhejiang Gongshang University - Statistics).} }

\author{ Li-Xin   ZHANG\footnote{Address: School of Statistics and Data Sciences, Zhejiang Gongshang University, 310018, P. R. China.  Email:stazlx@mail.zjgsu.edu.cn}\\
\\
 Zhejiang Gongshang University, P.R. China}

\date{}
\maketitle
\begin{abstract} Let $\{B_H(t);t\ge 0\}$  be  a fractional Brownian motion of order $H\in (0,1)$, and $J_{m,\bm\alpha}(B_H)$ be the $m$-fold weighted integrals of $B_H$ defined as
$$ J_{m,\bm\alpha}(B_H)(t) =\int_0^ts_m^{-\alpha_m}\int_0^{s_m}\cdots s_2^{-\alpha_2}\int_0^{s_2}s_1^{-\alpha_1}B_H(s_1)d s_1\; ds_2\cdots d s_m, $$
where $\alpha_1+\cdots+\alpha_i<H+i$, $i=1,\ldots,m$, $\bm\alpha=\bm\alpha_m=(\alpha_1,\ldots,\alpha_m)$.  
 We  show that
\begin{align*}
\liminf_{T\to \infty}     \frac{(\log\log T)^{H+m}}{T^{H+m-\alpha}}\sup_{0\le  t\le T}\left|\frac{ J_{m,\bm\alpha}(B_H)(t)}{t^{\alpha-\alpha_1-\cdots-\alpha_m}}\right|  
   = a_H\left( \frac{\kappa_{H+m}}{1-\alpha/(H+m)}\right)^{H+m}\;\; a.s.
\end{align*}
for all   $\alpha<H+m$,  and 
\begin{align*}
   \liminf_{T\to \infty} &  \sqrt{\frac{\log\log\log T}{\log T}}  \sup_{1\le t\le T}\left|\int_1^t  \frac{J_{m-1, \bm\alpha_{m-1}}(B_H)(s)}{s^{H+m-\alpha_1-\cdots-\alpha_{m-1}}}ds\right|\\
   &= \frac{\pi}{2}\frac{\sqrt{\beta(2H,1-H)}}{\prod_{i=1}^{m-1}\big(H+i-\alpha_1-\cdots-\alpha_i\big)}\;\; a.s.,
\end{align*}
 where $a_H$ is an explicit constant with $a_{\frac{1}{2}}=1$,  $\kappa_{\lambda}$ is a  constant which depends only on $\lambda$, and $\beta(a,b)$ is the beta function. In particular, the exact value of a Chung-type law of the iterated logarithm established by Duker,  Li and  Linde (2000) is found, and as an application, the Chung-type law of the iterated logarithm for the randomized play-the-winner rule is established.  
 The small ball probabilities of \(J_{m, \bm\alpha}(B_H)\) are established to show the liminf behaviors.   Similar Chung-type laws of the iterated logarithm and small ball probabilities for a  Riemann–Liouville fractional process are also established. 



\bigskip
This paper is submitted to Science in China-Mathematics. Some details (marked  red) are added in the present version for understanding  in  the proofs. 
\end{abstract}

\section{Introduction and main results} 
\setcounter{equation}{0}
Let $\{B_H(t);t\ge 0\}$   denote the $H$-fractional Brownian motion with $B_H(0)=0$ and $0<H<1$. 
Then $\{B_H(t); t \ge  0\}$ is a Gaussian process with mean zero and covariance function
$$ \ep\left[B_H(t)B_H(s)\right]=\frac{1}{2}\left(t^{2H}+s^{2H}-|t-s|^{2H}\right), $$
and $B_{\frac 12}(t)=B(t)$ is a standard Borwnian motion. Duker,  Li and  Linde \cite{DLL2000}  obtained the following Chung-type law of the iterated logarithm:
\begin{equation}\label{CLILforIB}
\liminf_{T\to \infty}  \frac{(\log\log T)^{3/2}}{T^{3/2-\alpha}}\sup_{0\le t\le T}\left|\int_0^t s^{-\alpha}B(s)ds\right|=C_{\alpha},
\end{equation}
where $0<C_{\alpha}<\infty$ and $\alpha<3/2$.    The exact  value of $ C_{\alpha} $ has remained unknown for a long time. This paper intends to investigate the same problem for the fractional Brownian motion and to determine the precise  limit value.

  One of the key step in the investigation of $B_H(t)$ is the following useful representation,
 $$ B_H(t)=a_H\big(W_H(t)+Z_H(t)\big), $$
 where
\begin{align}\label{eqdefaH}a_H=& \Gamma\big(H+1/2\big)\left(\frac{1}{2H}+\int_{-\infty}^0\left((1-s)^{H-1/2}-(-s)^{H-1/2}\right)^2ds\right)^{-1/2},  \\\label{eqdefZH}Z_H(t)=&\frac{1}{\Gamma\big(H+1/2\big)}\int_{-\infty}^0\left((t-s)^{H-1/2}-(-s)^{H-1/2}\right)dB(s),\;\; t\ge 0,\\
\label{eqdefW} & W_{\lambda}(t)=\frac{1}{\Gamma\big(\lambda+1/2\big)}\int_0^t (t-s)^{\lambda-1/2} dB(s),\;\; t\ge 0,
\end{align}
 $\lambda>0$ is a constant, and   $\Gamma(a)$ is a gamma function. $W_{\lambda}(t)$ is called the Riemann–Liouville fractional processes. Li and Linde \cite{LiLinde1998b} obtained the following small ball probability for $W_{\lambda}(t)$:
\begin{align}\label{eq:constant:1} \lim_{\epsilon\to 0} \epsilon^{1/\lambda}\log \pr\left(\sup_{0\le t\le 1}|W_{\lambda}(t)|<\epsilon\right)
=-\kappa_{\lambda} 
\end{align}
where $0<\kappa_{\lambda}=-\inf_{\epsilon>0}\epsilon^{1/\lambda}\log \pr\left(\sup_{0\le t\le 1}|W_{\lambda}(t)|<\epsilon\right)<\infty$. The value $\kappa_{\frac{1}{2}}=\frac{\pi^2}{8}$ is well
known.   Chen and Li \cite{ChenLi2003} proved that $\frac{3}{8} \leq \kappa_{\frac{3}{2}} \leq (2\pi)^{2/3} \cdot \frac{3}{8} $. It is clear that the constant $ C_{\alpha} $ in   \eqref{CLILforIB} is $ \kappa_{\frac{3}{2}}^{3/2} $ when $ \alpha=0 $, a result first obtained by Khoshnevisan and Shi \cite{KhoshnevisanShi1998}. We will prove that   $ C_{\alpha}=\left( \frac{\kappa_{\frac{3}{2}}}{1 - {2\alpha}/{3}} \right)^{3/2} $ for all $ \alpha < 3/2$.

 Let $\mathcal{C}[0,\infty)$ be the space of continuous functions. Define $I, I_m:\mathcal{C}[0,\infty)\to \mathcal{C}[0,\infty)$ as 
$$  I(w)(t)=\int_0^tw(s) ds,   \;\;
  I_0(w)(t)=w(t),$$
  $$ I_m(w)(t)=I(I_{m-1}(w))(t)=\frac{1}{(m-1)!}\int_0^t(t-s)^{m-1}w(s)ds,\;\; t\ge 0,\;\; m=1,2,\cdots. $$
For a real positive number $\gamma$, define the Riemann-Liouville fractional integral operator $I_{\gamma}$ of order $\gamma$ as
$$ I_{\gamma}w(t)=I_{\gamma}(w)(t)=\frac{1}{\Gamma(\gamma)}\int_0^t(t-s)^{\gamma-1}w(s)ds,\; \; t\ge 0. $$ 
Furthermore, for $\alpha$, $\alpha_0$ and $\bm \alpha=(\alpha_1,\ldots,\alpha_m)$, we define the weighted integral operators $J_{\alpha},  J_{m,\bm\alpha}$ as 
$$J_{\alpha}(w)(t)=\int_0^t s^{-\alpha}w(s)ds,  $$
$J_{m, \bm \alpha}=I_0$  for $m=0$, and
\begin{align*} 
J_{m,\bm \alpha}(w)(t)=&\int_0^ts_m^{-\alpha_m}\int_0^{s_m}\cdots s_2^{-\alpha_2}\int_0^{s_2}s_1^{-\alpha_1}w(s_1)d s_1\; ds_2\cdots  d s_m\\
=& J_{\alpha_m}\circ J_{\alpha_{m-1}}\circ \cdots \circ J_{\alpha_1}(w)(t)  \;  \text{for } m=1,2,\ldots. 
\end{align*}
Sometimes, we write $I_{\gamma}(w)(t)$ as $I_{\gamma}w(t)$ or $I_{\gamma}(w(t))$, and  $J_{m,\bm \alpha}(w)(t)$ as $J_{m,\bm \alpha}w(t)$ or $J_{m,\bm \alpha}(w(t))$. It is easily seen that $J_{m,\bm 0}=I_m$ and $I_{\gamma_1}\circ I_{\gamma_2}=I_{\gamma_1+\gamma_2}$.

Li and Linde \cite{LiLinde1998b}, Chen and Li \cite{ChenLi2003} and Li and Linde  \cite{LiLinde1999} obtained the small ball probabilities  and     Chung's laws of the iterated logarithm for the fractional Brownian motion $B_H$,  the $m$-fold integrated Brownian motion $I_m(B)$ and the Riemann-Liouville fractionally integrated fractional Brownian motion $I_{\gamma}(B_H)$, respectively. This paper   aims to explore the same problems for the $m$-fold weighted integrals of  $I_{\gamma}(B_H)$ and $W_{\lambda}$.  
\begin{theorem}\label{th:1:1}
For all  $m\ge 0$, $\gamma\ge 0$, $\alpha_1+\cdots+\alpha_i<H+\gamma+i$, $i=1,\ldots,m$, and $\alpha<H+\gamma+m$, we have that
\begin{align}\label{eq:th:1:1:1}
& \lim_{\epsilon\to 0} \epsilon^{1/(H+\gamma+m)}\log \pr\left(\sup_{0\le  t\le 1}\left|\frac{ J_{m,\bm\alpha}(I_{\gamma}B_H)(t)}{t^{\alpha-\alpha_1-\cdots-\alpha_m}}\right|<a_H \epsilon\right)\nonumber\\
=& \lim_{\epsilon\to 0} \epsilon^{1/(H+\gamma+m)}\log \pr\left(\sup_{0\le  t\le 1}\left|\frac{J_{m,\bm\alpha}(W_{H+\gamma})(t)}{t^{\alpha-\alpha_1-\cdots-\alpha_m}}\right|< \epsilon\right) 
=-\frac{\kappa_{H+\gamma+m}}{1-\alpha/(H+\gamma+m)}
\end{align}
and
\begin{align}\label{eq:th:1:1:2}
 &\liminf_{T\to \infty}  \frac{(\log\log T)^{H+\gamma+m}}{T^{H+\gamma+m-\alpha}}\sup_{0\le  t\le T}\left|\frac{ J_{m,\bm\alpha}(I_{\gamma}B_H)(t)}{t^{\alpha-\alpha_1-\cdots-\alpha_m}}\right|\nonumber\\
 =&a_H\liminf_{T\to \infty} \frac{(\log\log T)^{H+\gamma+m}}{T^{H+\gamma+m-\alpha}}\sup_{0\le  t\le T}\left|\frac{J_{m,\bm\alpha}(W_{H+\gamma})(t)}{t^{\alpha-\alpha_1-\cdots-\alpha_m}}\right|\nonumber\\
 = &a_H\left( \frac{\kappa_{H+\gamma+m}}{ 1-\alpha/(H+\gamma+m)}\right)^{H+\gamma+m}\;\; a.s.
\end{align}
\end{theorem}

The Chung-type laws of the iterated logarithm  \eqref{eq:th:1:1:2} are derived from the small ball estimates provided in \eqref{eq:th:1:1:1}, along with a recalling argument and the Borel-Cantelli Lemma. We will omit the proof since it is standard. When $H=1/2$, $\bm \alpha=\bm 0$ and $\gamma=0$, $J_{m,\bm\alpha}(I_{\gamma}B_H)=I_m(B)$ represents the $m$-fold integrated   Brownian motion,   and \eqref{eq:th:1:1:1}   was provided by Chen and Li \cite{ChenLi2003}  for  $\alpha=0$.  Additionally, when $m=0$, $\bm\alpha=\bm 0$ and $\gamma=0$, $J_{m,\bm\alpha}(I_{\gamma}B_H)=B_H$ is the fractional Brownian motion, and   \eqref{eq:th:1:1:1} was first provided for the special case of $\alpha=0$ by Li and Linde \cite{LiLinde1998b} and Shao \cite{Shao2003}, independently, and later by Lifshits and Linde \cite{LifshitsLinde2005},   who considered a general weight function $ w(t)$ instead of a specific form like $t^{-\alpha}$.  Furthermore, when $\bm\alpha=\bm 0$, $J_{m,\bm\alpha}(I_{\gamma}B_H)=I_{m+\gamma}(B_H)$, and \eqref{eq:th:1:1:1}   was established by Li and Linde \cite{LiLinde1999} for $\alpha=0$.

  The proof of \eqref{eq:th:1:1:1} will be given in Section \ref{sect:th:2}. Under the condition that $\alpha_1+\cdots+\alpha_i<H+\gamma+i$, $i=1,\ldots,m$, $J_{m,\bm\alpha}(I_{\gamma}B_H)$ and $J_{m,\bm\alpha}(W_{H+\gamma})$ are well defined and continuous self-similar Gaussian processes. In Section \ref{sect:th:2}, the small ball estimators of  weighted integrals of  a self-similar Gaussian process are also studied, and the precise  small ball probabilities of   $J_{m,\bm\alpha}(I_{\gamma}B_H)$ and $J_{m,\bm\alpha}(W_{H+\gamma})$ under $L^q$-norm  are obtained.

  It is a fascinating question to consider what occurs in \eqref{eq:th:1:1:2} when $\alpha = H +\gamma+ m$. In this case, the supremum must be taken away from zero; otherwise, the value of the supremum becomes infinite. The following theorem presents the results regarding the limit behaviors for $\alpha = H +\gamma+ m$.
  
  \begin{theorem}\label{th:1:2}   We have that 
\begin{align}\label{eq:th:1:2:1}
 \lim_{T\to \infty}  \Bigg\{ & \sup_{1\le t\le T}\bigg|  \frac{J_{m,\bm\alpha}(I_{\gamma}B_H)(t)}{t^{H+\gamma+m-\alpha_1-\cdots-\alpha_m}}\bigg|-\sigma_{B,m,\bm \alpha,\gamma}\sqrt{2\log\log T}\Bigg\}=0\;\; a.s. 
  \end{align} 
for all $m\ge 0$,   $\gamma\ge 0$ and  $\alpha_1+\cdots+\alpha_i<H+\gamma+i$, $i=1,\ldots,m$; and
\begin{align}
  \label{eq:th:1:2:3}
 \lim_{T\to \infty}  \Bigg\{ & \sup_{1\le t\le T}\bigg|  \frac{J_{m,\bm\alpha}(W_{\lambda})(t)}{t^{\lambda+m-\alpha_1-\cdots-\alpha_m}}\bigg|-\sigma_{W,m,\bm \alpha,\lambda}\sqrt{2\log\log T}\Bigg\}=0\;\; a.s.
\end{align}
for all $m\ge 0$,   $\lambda> 0$ and  $\alpha_1+\cdots+\alpha_i<\lambda+i$, $i=1,\ldots,m$, where
$$\sigma_{B,m,\bm \alpha,\gamma}^2= \Var\left\{J_{m,\bm\alpha}(I_{\gamma}B_H)(1)\right\}, $$
$$\sigma_{W,m,\bm \alpha,\lambda}^2= \Var\left\{J_{m,\bm\alpha}(W_{\lambda})(1)\right\} $$
and
$$\sigma_{B,0, 0,\gamma}^2=\frac{1}{\Gamma^2(\gamma)}\int_0^1\int_0^1x^{2H+1}[(1-x)(1-xy)]^{\gamma-1}\left[1+y^{2H}-(1-y)^{2H}\right]dxdy,$$
$$\sigma_{W,0, 0,\lambda}^2= \Var\left\{ W_{\lambda}(1) \right\}=\frac{1}{2\lambda \Gamma^2\big(\lambda+1/2\big)}.$$

\end{theorem} 
\begin{remark}\label{rk:1} 
When $H\ne 1/2$, $\sigma_{B,m,\bm \alpha,\gamma}>a_H\sigma_{W,m,\bm \alpha, H+\gamma}$. If $H=1/2$, then $a_H=1$, $\sigma_{B,m,\bm \alpha,\gamma}=\sigma_{W,m,\bm \alpha, H+\gamma}$ and $\sigma_{B,m,\bm 0,\gamma}=\sigma_{W,m,\bm 0, H+\gamma}=\sigma_{W,0, 0, H+\gamma+m}$. 
\end{remark}

The proof of Theorem \ref{th:1:2} will be given in Section \ref{sect:th:4}. If $H=1/2$, $m=1$ and $\gamma=0$, then $\sigma_{B,m, \alpha_1,\gamma}^2=\frac{2}{(2-\alpha_1)(3-2\alpha_1)}$,  $a_H=1$, and 
$W_{H+2}(t)=\int_0^t B(s)ds $
is the integrated Brownian motion.  Thus, we have the following corollary.

\begin{corollary}\label{cor:1.1} Let $B(t)$ be a standard Brownian motion. For all  $\alpha<3/2$, we have that
\begin{equation}\label{eq:cor:1:1}
 \liminf_{T\to \infty} \frac{(\log\log T)^{3/2}}{T^{3/2-\beta}}  \sup_{0\le t\le T}\left| \frac{\int_0^t  s^{-\alpha}B(s)ds}{t^{\beta-\alpha}}\right|= \left( \frac{\kappa_{\frac{3}{2}}}{ 1-2\beta/3}\right)^{3/2}\;\; a.s.,\ 
\end{equation}
\begin{equation}\label{eq:cor:1:2}
 \lim_{\epsilon\to 0}\epsilon^{2/3}\log\pr\left( \sup_{0\le t\le 1}\left| \frac{\int_0^t  s^{-\alpha}B(s)ds}{t^{\beta-\alpha}}\right|<\epsilon\right)= - \frac{\kappa_{\frac{3}{2}}}{ 1-2\beta/3}, 
\end{equation}
for all $\beta<3/2$, and
\begin{equation}\label{eq:cor:1:3}
 \lim_{T\to \infty} \frac{1}{\sqrt{\log\log T}}  \sup_{1\le t\le T}\left| \frac{\int_0^t  s^{-\alpha}B(s)ds}{t^{3/2-\alpha}}\right|= 
 \frac{2}{\sqrt{(2-\alpha)(3-2\alpha)}}\;\; a.s. \\
\end{equation}
\end{corollary} 
\eqref{eq:cor:1:3} relates to the limit law of the iterated logarithm established by Chen \cite{ChenX2015} and Robbins and Siegmund \cite{RS1972}.  According to \eqref{eq:cor:1:1}, the constant in \eqref{CLILforIB} is given by $ C_{\alpha} = \left( \frac{\kappa_{\frac{3}{2}}}{1 - {2\alpha}/{3}} \right)^{3/2} $.  In the last section, we will give an application of \eqref{eq:cor:1:1} and \eqref{eq:cor:1:2} to an urn model called the randomized play-the-winner rule.

  When $\alpha = 3/2$,   \eqref{CLILforIB} fails to hold. In this case, since the integral $\int_0^t s^{-\alpha}B(s)ds$ is not finite, we consider   $\int_1^t s^{-\alpha}B(s)ds$ instead.  Notice that   $\int_1^t s^{-3/2}B(s)ds=2\int_1^ts^{-1/2} d B(s)+2 B(1)-2 t^{-1/2}B(t)$,
$$\limsup_{T\to \infty} \frac{1}{\sqrt{2\log\log T}}\sup_{1\le t\le T} t^{-1/2}|B(t)| =1 \;\; a.s. $$
and 
$$ \left\{\int_1^ts^{-1/2} d B(s); t\ge 1\right\}\overset{\mathcal{D}}=\left\{B(\log t); t\ge 1\right\}. $$
We have that
\begin{align*}
  & \liminf_{T\to \infty} \sqrt{\frac{\log\log\log T}{\log T}}\sup_{1\le t\le T}\left|\int_1^t s^{-3/2} B(s)ds\right|\\
  =& 
 2\liminf_{S\to \infty} \sqrt{\frac{\log\log S }{S}}\sup_{0\le s\le S}\left| B(s)\right|=
\frac{\pi}{\sqrt{2}} \;\; a.s.
\end{align*}
In general, we have the following theorem. 
\begin{theorem}\label{th:1:3} We have that
\begin{align}\label{eq:th:1:3:1}
   \liminf_{T\to \infty}   &\sqrt{\frac{\log\log\log T}{\log T}}  \sup_{1\le t\le T}\left|\int_1^t  \frac{J_{m, \bm\alpha}(I_{\gamma}B_H)(s)}{s^{H+\gamma+m+1-\alpha_1-\cdots-\alpha_m}}ds\right|\nonumber\\
   &= \frac{\pi}{2} \frac{\sqrt{\beta(2H,1-H)}}{\prod_{i=1}^m\big(H+\gamma+i-\alpha_1-\cdots-\alpha_i\big)}\frac{\Gamma(H+1)}{\Gamma(\gamma+1+H)} \;\; a.s.
   \end{align}
   for all $m\ge 0$, $\gamma\ge 0$ and   $\alpha_1+\cdots+\alpha_i<H+\gamma+i$, $i=1,\ldots,m$, and
    \begin{align}
   \label{eq:th:1:3:2}
   \liminf_{T\to \infty}   &\sqrt{\frac{\log\log\log T}{\log T}}  \sup_{1\le t\le T}\left|\int_1^t  \frac{J_{m, \bm\alpha}(W_{\lambda})(s)}{s^{\lambda+m+1-\alpha_1-\cdots-\alpha_m}}ds\right| \nonumber\\
   &= \frac{\pi}{\sqrt{8}}\frac{1}{\prod_{i=1}^m\big(\lambda+i-\alpha_1-\cdots-\alpha_i\big)}\frac{\Gamma(1/2)}{\Gamma(\lambda+1)}\;\; a.s.,
\end{align}
 for all $m\ge 0$, $\lambda> 0$ and   $\alpha_1+\cdots+\alpha_i<\lambda+i$, $i=1,\ldots,m$, 
where $\beta(a,b)$ is the beta function. 
In particular, we have that 
$$   \liminf_{T\to \infty}\sqrt{\frac{\log\log\log T}{\log T}}\sup_{1\le t\le T}\left|\int_1^t s^{-(3/2+m)} I_{m}(B)(s)ds\right|=\frac{\pi}{\sqrt{8}}\frac{2^{m+1}}{(2m+1)!!} \;\; a.s.$$
 
\end{theorem}
 Theorem \ref{th:1:3} will be proved in Section \ref{sect:th:4} by an almost sure invariance principle of Shao \cite{Shao1985}. 

  The following corollary summarizes Chung-type laws of the iterated logarithm of $\int_1^t s^{-\alpha} B_H(s) ds$. 
\begin{corollary} We have that
\begin{align*}
 \liminf_{T\to \infty} & \frac{(\log\log T)^{H+1}}{T^{H+1-\alpha}}\sup_{1\le t\le T}\left|\int_1^t \frac{B_H(s)}{s^{\alpha}}ds\right|= a_H\left( \frac{\kappa_{H+1}}{ 1- \alpha/(H+1)}\right)^{H+1} a.s.,\;\alpha<H+1,\\
   \liminf_{T\to \infty} & \sqrt{\frac{\log\log\log T}{\log T}}\sup_{1\le t\le T}\left|\int_1^t\frac{B_H(s)}{s^{\alpha}}ds\right|=\frac{\pi}{2}\sqrt{\beta(2H,1-H)}\;\; a.s.,\;\alpha=H+1.
\end{align*}
\end{corollary}

  \section{Laws of the iterated logarithm}
  \label{sect:th:4}
\setcounter{equation}{0}

 \eqref{eq:th:1:1:2} follows from \eqref{eq:th:1:1:1}.  Next we consider the proofs of  Theorems \ref{th:1:2} and \ref{th:1:3}. We say that a Gaussian process  $  X(t)$ is  self-similar   of index $\tau>0$,  if
$$ \{  X(ct); t\ge 0\}\overset{\mathscr{D}}=\{c^{\tau} X(t); t\ge 0\}, \;\; c>0.  $$
It is easily seen that $W_{\lambda}(t)$ is self-similar of index $\lambda$, and $J_{m,\bm\alpha}(I_{\gamma}B_H)(t)$ and $J_{m,\bm\alpha}(W_{H+\gamma})(t)$ are  self-similar   of index $H+\gamma+m-\alpha_1-\cdots-\alpha_m$. If $X(t)$ is a self-similar Gaussian process of index $\tau>0$, then $\frac{X(e^t)}{e^{\tau t}}$ is a stationary Gaussian process. 
To prove the limit laws of the iterated logarithm, as given in Theorem \ref{th:1:2}, we need a  lemma on the limit behavior of a stationary Gaussian process. 
 
\begin{lemma} (\cite[Theorem 5.4]{Pickans1967}\label{lem:Pickans}) Let $U(t)$ be a centered stationary       Gaussian process with the covariance function
$ r(h)=\ep[U(t)U(t+h)]$. If
$$ \exists \epsilon>0, \;\;\limsup_{|h|\to 0}|h|^{-\epsilon}\{r(0)-r(h)\}<\infty, $$
and $\lim_{|h|\to \infty}r(h)\log |h|=0$, then
$$ \sup_{0\le s\le t}U(s)-\sqrt{2r(0)\log t}\to 0 \;\; a.s. \text{ as } t\to \infty. $$
\end{lemma}
   
{\bf Proof of Theorem \ref{th:1:2}.}     We first consider   \eqref{eq:th:1:2:1}. Let $\alpha_0=0$. For  $\alpha_1+\cdots+\alpha_i<H+\gamma+i$, $i=1,\ldots, m$,  let 
\begin{equation}\label{eq:proofth1:2:OU1}
\begin{aligned}   X_{0,0}(t)=I_{\gamma}(B_H)(t), &\;\; U_{0,0}(t)=\frac{  X_{0,0}(e^t)}{e^{(H+\gamma)t}},  \\
  X_{m,\bm \alpha}(t)= J_{m,\bm\alpha}(I_{\gamma}B_H)(t), &\;\; U_{m,\bm \alpha}(t)=\frac{  X_{m,\bm \alpha}(e^t)}{e^{(H+\gamma+m-\alpha_1-\cdots-\alpha_m)t}}. 
\end{aligned}
\end{equation}
Then 
$$ \sup_{1\le t\le T}t^{-(H+\gamma+m-\alpha_1-\cdots-\alpha_m)}|X_{m, \bm \alpha}(t)|=\sup_{0\le t\le \log T}|U_{m, \bm \alpha}(t)|. $$
Since $X_{m,\bm \alpha}(t)$ is a centered, self-similar Gaussian process of index $H+\gamma+m-\alpha_1-\cdots-\alpha_m$, $U_{m,\bm\alpha}(t)$ is a stationary Gaussian process with the covariance function
\begin{align}\label{eq:proofth1:2:OU2} r(h)=& r_{m,\bm \alpha,\gamma}(h)=:\ep\left[U_{m,\bm \alpha}(t)U_{m,\bm \alpha}(t+h)\right]\nonumber\\
=& \frac{\ep[X_{m,\bm \alpha}(e^t)X_{m,\bm \alpha}(e^{t+h})]}{e^{(H+\gamma+m-\alpha_1-\cdots-\alpha_m)(2t+h)}}
=\frac{\ep[X_{m,\bm \alpha}(1)X_{m,\bm \alpha}(e^h)]}{e^{(H+\gamma+m-\alpha_1-\cdots-\alpha_m)h}}.
\end{align}
It is sufficient to show that $r(t)$ satisfies the conditions of Lemma \ref{lem:Pickans}. When $\alpha_0=0$, $m=0$, and $\gamma=0$,  $X_{0,0}(t)=B_H(t)$. Thus 
$$ r_{0,0,0}(h)=\frac{1}{2}\left\{e^{Hh}+e^{-Hh}-\Big|e^{h/2}-e^{-h/2}\Big|^{2H}\right\}. $$
It is can be shown  that 
\begin{equation}\label{eq:proofth1:2:ad1} |r_{0,0,0}(t+h)-r_{0,0,0}(t)|\le C_H|h|^{(2H)\wedge 1}, \;\; 0< r_{0,0,0}(h)\le C_H\exp\{-(H\wedge (1-H))|h|\}. 
\end{equation}
\red{In fact,  when $H=1/2$,   $r_{0,0,0}(t)=e^{-|t|/2}$ and \eqref{eq:proofth1:2:ad1} is obvious. In general, without loss of generality, we assume $t\ge 0$. Notice  $0\le 1-(1-x)^{\alpha}\le c_{\alpha} x$, $0\le x\le 1$, if $\alpha\ge 0$. Thus,
\begin{align*} 0<r_{0,0,0}(t)=& \frac{1}{2}e^{-Ht}+\frac{1}{2}e^{Ht}\big(1-(1-e^{-t})^{2H}\big) \\
\le & \frac{1}{2}e^{-Ht}+\frac{1}{2}C_He^{Ht}e^{-t}\le C_H\exp\left\{-(H\wedge(1-H))t\right\}. 
\end{align*} 
Next, we consider $r_{0,0,0}(t+h)-r_{0,0,0}(t)$. Notice $|r_{0,0,0}(t)|\le C_H$. Without loss of generality, we assume $t\ge 0$ and $|h|\le 1$. 
Then 
\begin{align*}
&r_{0,0,0}(t+h)-r_{0,0,0}(t)\\
=& (1-e^{-Hh})\gamma(t+h)+\frac{1}{2}e^{-Ht}(e^{-2Hh}-1) +
\frac{1}{2} e^{Ht}\left(|1-e^{-t}|^{2H}-|1-e^{-t-h}|^{2H}\right).
\end{align*}
Notice, $|r_{0,0,0}(t+h)|\le C_H$, $|1-e^{-Hh}|\le C_H |h|$ for $|h|\le 1$, and
\begin{align*}
& \left||1-e^{-t}|^{2H}-|1-e^{-t-h}|^{2H}\right|\le C_H|e^{-t}-e^{-t-h}|\le C_H e^{-t}|h| \text{ when } 2H\ge 1  \\
& \Big( \text{ since the deviation of } |1-x|^{2H} \text{ is bounded  for } x\in[0,e] \text{ when }    2H\ge 1\Big);\\
& \left||1-e^{-t}|^{2H}-|1-e^{-t-h}|^{2H}\right|\le |e^{-t}-e^{-t-h}|^{2H}\le C_H e^{-2Ht}|h|^{2H}  \text{ when } 2H\le  1\\
&\Big(\text{ since } |x+y|^{2H}\le |x|^{2H}+|y|^{2H} \text{ when } 0\le 2H\le 1\Big). 
\end{align*}
It follows that, when $2H\ge 1$, 
$$|r_{0,0,0}(t+h)-r_{0,0,0}(t)|\le C_H|h|+C_H e^{(H-1)t}|h|\le C_H |h|, $$ 
and, when $2H\le 1$, 
$$|r_{0,0,0}(t+h)-r_{0,0,0}(t)|\le C_H|h|+C_He^{-Ht}|h|^{2H}\le C_H|h|^{2H}. $$
\eqref{eq:proofth1:2:ad1} is proven. }
Thus $r_{0,0,0}(h)$ satisfies the conditions of Lemma \ref{lem:Pickans}.  When  $\alpha_0=0$, $m=0$, and $\gamma>0$,
\begin{align*}
&r_{0,0,\gamma}(h)=\frac{\ep\left[I_{\gamma}(B_H)(1)I_{\gamma}(B_H)(e^h)\right]}{e^{(H+\gamma)h}}\\
=&\frac{e^{-(H+\gamma)h}}{\Gamma^2(\gamma)}\int_0^{e^h}\int_0^1(e^h-u)^{\gamma-1}(1-v)^{\gamma-1}\ep[B_H(v)B_H(u)]dvdu\\
=&\frac{e^{-(H+\gamma)h}}{\Gamma^2(\gamma)}\int_{-\infty}^{0}\int_{-\infty}^0(e^h-e^{u+h})^{\gamma-1}(1-e^v)^{\gamma-1}\ep[B_H(e^v)B_H(e^{u+h})]e^v e^{u+h}dvdu\\
=&\frac{1}{\Gamma^2(\gamma)}\int_{-\infty}^{0}\int_{-\infty}^0 r_{0,0,0}(h+u-v)(1-e^u)^{\gamma-1}e^{u(H+1)}(1-e^v)^{\gamma-1}e^{v(H+1)} dvdu.
\end{align*}
It follows that
\begin{align*}
|r_{0,0,\gamma}(t+h)-r_{0,0,\gamma}(t)|\le & C_H|h|^{(2H)\wedge 1}\frac{1}{\Gamma^2(\gamma)}\left(\int_{-\infty}^{0}(1-e^u)^{\gamma-1}e^{u(H+1)}du\right)^2\\
\le &  C_H\frac{\beta^2(\gamma,H+1)}{\Gamma^2(\gamma)}|h|^{(2H)\wedge 1}
\end{align*}
and
\begin{align*}
&0< r_{0,0,\gamma}(h)\\
\le &C_H\frac{1}{\Gamma^2(\gamma)}\int_{-\infty}^{0}\int_{-\infty}^0 e^{-(H\wedge (1-H))(|h|-|v|-|u|)}(1-e^u)^{\gamma-1}e^{u(H+1)}(1-e^v)^{\gamma-1}e^{v(H+1)} dvdu\\
\le & C_H\frac{1}{\Gamma^2(\gamma+1)} e^{-(H\wedge (1-H))|h|}.
\end{align*}
Thus $r_{0,0,\gamma}(h)$ satisfies the conditions of Lemma \ref{lem:Pickans}, and
\begin{align*}
&\sigma^2_{B,0, 0,\gamma}=r_{0,0,\gamma}(0)\\
=& \frac{1}{\Gamma^2(\gamma)}\int_{-\infty}^{0}\int_{-\infty}^0 r_{0,0,0}(u-v)(1-e^u)^{\gamma-1}e^{u(H+1)}(1-e^v)^{\gamma-1}e^{v(H+1)} dvdu\\
=&\frac{1}{\Gamma^2(\gamma)}\int_0^1\int_0^1x^{2H+1}[(1-x)(1-xy)]^{\gamma-1}\left[1+y^{2H}-(1-y)^{2H}\right]dxdy.
\end{align*}

For $m\ge 1$, it is easily verified that
$$U_{m,\bm \alpha}(t)=\int_{-\infty}^0 U_{m-1, \bm \alpha_{m-1}}(u+t)e^{(H+\gamma+m-\alpha_1-\cdots-\alpha_m)u}du, $$
\begin{equation}\label{eq:proofth:4:1} r_{m,\bm \alpha,\gamma}(h)=\int_{-\infty}^0\int_{-\infty}^0 r_{m-1, \bm \alpha_{m-1},\gamma}(h+u-v)e^{(H+\gamma+m-\alpha_1-\cdots-\alpha_m)(u+v)} dudv,
\end{equation}
where $\bm\alpha=\bm\alpha_m=(\alpha_1,\ldots,\alpha_{m})$ and $\bm\alpha_{m-1}=(\alpha_1,\ldots,\alpha_{m-1})$. By the induction, we can find positive constants $C_{m, \bm\alpha,\gamma}$,  $c_{m,\bm \alpha,\gamma}$ and $\epsilon=(2H)\wedge 1$ such that
\begin{equation}\label{eq:proofth:4:2}
\begin{aligned} & |r_{m,\bm \alpha,\gamma}(t+h)-r_{m,\bm \alpha,\gamma}(t)|\le    C_{m,\bm \alpha,\gamma}|h|^{\epsilon}, \\
& 0< r_{m,\bm \alpha,\gamma}(h)\le  C_{m,\bm \alpha,\gamma}\exp\{- c_{m,\bm \alpha,\gamma}|h|\}.
 \end{aligned}
\end{equation}
 Thus, $r_{m,\bm \alpha}$ satisfies the conditions of Lemma \ref{lem:Pickans}. 
 
 For  \eqref{eq:th:1:2:3}, with the same arguments of the proof of \eqref{eq:th:1:2:1}, it is sufficient to show that \eqref{eq:proofth:4:2} is satisfied in  the case of $m=0$. Let 
\begin{equation} \label{eq:proofth1:2:Ulambda1} U_{\lambda}(t)=\frac{W_{\lambda}(e^t)}{e^{\lambda t}}. 
\end{equation}
Then
$$ \sup_{1\le t\le T}t^{-\lambda} |W_{\lambda}(t)|=\sup_{0\le t\le \log T}|U_{\lambda}(t)|, $$
and $U_{\lambda}(t)$ is a   stationary Gaussian process with the covariance function
\begin{align} \label{eq:proofth1:2:Ulambda2}   r_{\lambda}(h)= & \ep\left[U_{\lambda}(t)U_{\lambda}(t+h)\right]
=\frac{\ep[W_{\lambda}(e^t)W_{\lambda}(e^{t+h})]}{e^{\lambda (2t+ h)}}
=\frac{\ep[W_{\lambda}(1)W_{\lambda}(e^h)]}{e^{\lambda h}} \nonumber\\
=&\frac{1}{\Gamma^2(\lambda+1/2)}e^{-\lambda h}\int_0^1(e^h-x)^{\lambda-1/2}(1-x)^{\lambda-1/2}dx, \;\; h\ge 0.
\end{align}
When $\lambda= 1/2$, $r_{\lambda}(h)=e^{-h/2}$, $h\ge 0$. When $\lambda>1/2$ and $h\ge 0$,
$$r_{\lambda}(h)\le \frac{1}{\Gamma^2(\lambda+1/2)}e^{-\lambda h}\int_0^1e^{(\lambda-1/2)h}(1-x)^{\lambda-1/2}dx=e^{-h/2}C_{\lambda},
$$
$$r_{\lambda}(h)\ge \frac{1}{\Gamma^2(\lambda+1/2)}e^{-\lambda h}\int_0^1(1-x)^{(2\lambda-1)}dx=e^{-\lambda h}r_{\lambda}(0).
$$
When $\lambda<1/2$ and $h\ge 0$,
$$r_{\lambda}(h)\le \frac{1}{\Gamma^2(\lambda+1/2)}e^{-\lambda h}\int_0^1(1-x)^{(2\lambda-1)}dx=e^{-\lambda h}r_{\lambda}(0).
$$
\begin{align*}
r_{\lambda}(h)\ge &   \frac{1}{\Gamma^2(\lambda+1/2)}e^{-\lambda h}\int_0^1(e^h-x)^{(2\lambda-1)}dx \\
= & r_{\lambda}(0)e^{-\lambda h}\left(e^{2\lambda h}-(e^h-1)^{2\lambda}\right)
\ge r_{\lambda}(0) \left(1-(1-e^{-h})^{2\lambda}\right).
\end{align*}
It follows that $0\le r_{\lambda}(h)\le C_{\lambda} e^{-(\lambda\wedge \frac{1}{2}) |h|}$ and $0\le r_{\lambda}(0)-r_{\lambda}(h)
\le C_{\lambda} |h|^{(2\lambda)\wedge 1}$. Furthermore,
\begin{align*}
&\big|r_{\lambda}(t+h)-r_{\lambda}(t)\big|=\Big|\ep\big[U_{\lambda}(0)[U_{\lambda}(t+h)-U_{\lambda}(t)]\big]\Big|\\
\le & \Big(\ep U_{\lambda}^2(0)\cdot \ep\big(U_{\lambda}(t+h)-U_{\lambda}(t)\big)^2 \Big)^{1/2}
=\Big(r_{\lambda}(0)\cdot 2\big(r_{\lambda}(0)-r_{\lambda}(h)\big)  \Big)^{1/2}\le  C_{\lambda} |h|^{\lambda\wedge (1/2)}.
\end{align*}
Thus, $r_{\lambda}(t)$ satisfies \eqref{eq:proofth:4:2} and the conditions of Lemma \ref{lem:Pickans}. Hence, 
 \eqref{eq:th:1:2:3} holds with
 \begin{align*}
 &\sigma_{W,0,0,\lambda}^2=\Var\{W_{\lambda}(1)\}=r_{\lambda}(0)\\
 =&\frac{1}{\Gamma^2(\lambda+1/2)}e^{-\lambda h}\int_0^1(1-x)^{2\lambda-1}dx=\frac{1}{2\lambda\Gamma^2(\lambda+1/2)}.
 \end{align*}
The proof of Theorem \ref{th:1:2} is completed.  $\Box$

 Notice that
$ I_{\gamma}(B_H)(t)=a_HW_{H+\gamma}(t)+a_HI_{\gamma}(Z_H)(t), $
and that $W_{H+\gamma}(t)$ and $Z_H(t)$ are independent. Thus
$$\Var\{J_{m,\bm\alpha}(I_{\gamma}B_H)(t)\}=a_H^2\Var\{J_{m,\bm\alpha}(W_{H+\gamma})(t)\} + a_H^2\Var\{J_{m,\bm\alpha}(I_{\gamma}Z_H)(t)\}. $$
Thus, the conclusions in Remark \ref{rk:1} follow. Finally, we consider  Theorem \ref{th:1:3}. 

{\bf Proof of Theorem \ref{th:1:3}.} Let $U_{m,\bm \alpha}(t)$ and $r_{m,\bm \alpha,\gamma}(h)$ be defined as \eqref{eq:proofth1:2:OU1} and \eqref{eq:proofth1:2:OU2}, respectively. Notice that
$$\int_1^t s^{-(H+\gamma+m+1-\alpha_1-\cdots-\alpha_m)}J_{m,\bm\alpha}(I_{\gamma}B_H)(s)ds=\int_0^{\log t}U_{m,\bm \alpha}(x)dx. $$
Write $V(t)=\int_0^tU_{m,\bm \alpha}(x)dx$ and $\psi(h)=\int_0^hr_{m,\bm \alpha,\gamma}(s)ds$. Then by \eqref{eq:proofth:4:1} and \eqref{eq:proofth:4:2},
\begin{align*} &\ep\left[V(t)V(t+h)\right]= \int_0^{t+h}d u \int_0^t r_{m,\bm \alpha,\gamma}(|u-v|)dv\\
=&2\int_0^t\int_0^u r_{m,\bm \alpha,\gamma}(u-v)dvdu+\int_t^{t+h}du \int_0^t r_{m,\bm \alpha,\gamma}(u-v)dv\\
=&2\int_0^t\psi(u)du+\int_0^h\big[\psi(u+t)-\psi(u)\big]du\\
=&2t\int_0^{\infty}r_{m,\bm \alpha,\gamma}(v)dv +O(1) \text{ as } t\to \infty \text{ uniformly in } h\ge 0.   
\end{align*}
Thus, it is expected that $V(t)$ behaves as a Brownian motion. Actually, we let
$$ X_n=V(n)-V(n-1)=\int_{n-1}^n U_{m,\bm \alpha}(x)dx, \;\; n=1,2,\ldots. $$
Then $\{X_n;n\ge 1\}$ is a sequence of stationary Gaussian random variables with 
$$ \gamma(n)=|\ep X_{k+n} X_k|=\left|\int_{n}^{n+1}\int_{0}^1 r_{m,\bm \alpha}(u-v)dv du\right|\le Ce^{-cn} $$
by \eqref{eq:proofth:4:2} and the fact that $\{U_{m,\bm \alpha}(t);t\ge 0\}$ is a stationary Gaussian process.  Applying an almost sure invariance principle of Shao \cite{Shao1985} (c.f. Corollary 14.2.1 of Lin and Lu \cite{LinLu1997}), we can find  a standard Browian motion $W(t)$ such that
$$ V(n)-\widetilde{\sigma} W(n)=O\big(\log^{1/2} n\big)\; a.s., $$
where 
$$\widetilde{\sigma}^2=\widetilde{\sigma}^2_{m,\bm \alpha,\gamma,H}=\ep X_1^2+2\sum_{k=2}^{\infty} \ep X_1X_k=\lim_{n\to \infty} \frac{\ep V^2(n)}{n}=2\int_0^{\infty}r_{m,\bm \alpha,\gamma}(v)dv. $$
It can be checked that
$$ \sup_{n\le t\le n+1}|V(t)-V(n)|=O\big(\log^{1/2} n\big)\; a.s. \;\text{ and }  \sup_{n\le t\le n+1}|W(t)-W(n)|=O\big(\log^{1/2} n\big)\; a.s. $$
It follows that
$$ V(t)-\widetilde{\sigma} W(t)=O\big(\log^{1/2} t\big)\; a.s. \text{ as } t\to \infty. $$
Hence, 
\begin{align*}
   &\liminf_{T\to \infty}  \sqrt{\frac{\log\log\log T}{\log T}} \sup_{1\le t\le T}\left|\int_1^t  \frac{J_{\bm\alpha, m}(B_H)(s)}{s^{H+m+1-\alpha_1-\cdots-\alpha_m}}ds\right|\\
   =& \liminf_{S\to \infty}   \sqrt{\frac{\log\log S}{S}}  \sup_{0\le t\le S}|V(t)| 
   = \widetilde{\sigma}\liminf_{S\to \infty}  \sqrt{\frac{\log\log S}{S}}  \sup_{0\le t\le S}|W(t)| 
   = \frac{\pi}{\sqrt{8}}\widetilde{\sigma}\;\; a.s.
\end{align*} 
When $m=0$ and $\gamma=0$, 
\begin{align*}
\widetilde{\sigma}^2=&\widetilde{\sigma}^2_{m,\bm \alpha,\gamma,H}=\widetilde{\sigma}^2_{0,0,0,H}=2\int_0^{\infty}r_{ 0,0,0}(v)dv=\frac{1}{H}+ \int_0^{\infty}\left[1-(1-e^{-h})^{2H}\right]e^{Hh}dh\\
=&\frac{1}{H}+\frac{1}{H}\int_0^{\infty}\left[1-(1-e^{-h})^{2H}\right] d e^{Hh}
=2\int_0^{\infty}e^{Hh}(1-e^{-h})^{2H-1}e^{-h}dh\\
=& 2\int_0^1 u^{-H}(1-u)^{2H-1}du=2\beta(2H,1-H).
\end{align*}  
When $m=0$ and $\gamma>0$, 
\begin{align*}
&\widetilde{\sigma}^2_{m,\bm \alpha,\gamma, H}= \widetilde{\sigma}^2_{0,0,\gamma, H}=2\int_0^{\infty}r_{0,0,\gamma}(h)dh=\int_{-\infty}^{\infty}r_{0,0,\gamma}(h)dh\\
=&\frac{1}{\Gamma^2(\gamma)}\int_{-\infty}^{0}\int_{\infty}^0\left[\int_{-\infty}^{\infty} r_{0,0,0}(t+h-s)dh\right](1-e^t)^{\gamma-1}e^{t(H+1)}(1-e^s)^{\gamma-1}e^{s(H+1)} dsdt\\
=& \widetilde{\sigma}^2_{0,0,0, H}\frac{\beta^2(\gamma,H+1)}{\Gamma^2(\gamma)}=2\beta(2H,1-H)\frac{\Gamma^2(H+1)}{\Gamma^2(H+1+\gamma)}.
\end{align*}
When $m\ge 1$, by \eqref{eq:proofth:4:1}, it follows that
\begin{align*}
&\widetilde{\sigma}^2_{m,\bm \alpha,\gamma,H}= 2\int_0^{\infty}r_{m,\bm \alpha,\gamma}(h)dh=\int_{-\infty} ^{\infty}r_{m,\bm \alpha,\gamma}(h)dh\\
 =&\int_{-\infty}^0\int_{-\infty}^0 \left[\int_{-\infty}^{\infty}r_{m-1, \bm \alpha_{m-1},\gamma}(h+u-v)dh\right]e^{(H+m+\gamma-\alpha_1-\cdots-\alpha_m)(u+v)}  dudv\\
 =&\int_{-\infty}^0\int_{-\infty}^0 \widetilde{\sigma}^2_{m-1, \bm\alpha_{m-1},\gamma, H} e^{(H+m+\gamma-\alpha_1-\cdots-\alpha_m)(u+v)} dudv \\
 =&  \frac{\widetilde{\sigma}^2_{m-1,\bm\alpha_{m-1}, H}}{\big(H+m+\gamma-\alpha_1-\cdots-\alpha_m\big)^2} 
 =\cdots=\frac{\widetilde{\sigma}^2_{0,0,\gamma,H}}{\prod_{i=1}^m\big(H+\gamma+i-\alpha_1-\cdots-\alpha_i\big)^2}\\
 =&\frac{2\beta(2H,1-H)}{\prod_{i=1}^m\big(H+\gamma+i-\alpha_1-\cdots-\alpha_i\big)^2}\frac{\Gamma^2(H+1)}{\Gamma^2(H+1+\gamma)}.
\end{align*}
In particular, if $H=1/2$, $\bm \alpha=\bm0$ and $\gamma=0$, then $J_{m,\bm \alpha}(I_{\gamma}B_H)=I_{m}(B)$ and 
\begin{align*}
\widetilde{\sigma}^2_{m,\bm \alpha,\gamma,H}=\frac{4}{\prod_{i=1}^m\big(\frac{1}{2}+i\big)^2}
= \left(\frac{2^{m+1}}{(2m+1)!!}\right)^2. 
\end{align*}
The proof 
\eqref{eq:th:1:3:1} is completed. 

The proof of \eqref{eq:th:1:3:2} is similar. It is sufficient to notice that, for $m=0$,  
\begin{align*}
  \widetilde{\sigma}^2=&2\int_0^{\infty}r_{\lambda}(v)dv 
=\frac{2}{\Gamma^2(\lambda+1/2}\int_0^1(1-x)^{\lambda-1/2}\int_0^{\infty}  e^{-v/2}(1-xe^{-v})^{\lambda-1/2}dvdx\\
=&\frac{2}{\Gamma^2(\lambda+1/2)}\int_0^1x^{-1/2}(1-x)^{\lambda-1/2}\int_0^x  u^{-1/2} (1-u)^{\lambda-1/2}dudx\\
=&\frac{2}{\Gamma^2(\lambda+1/2)}\frac{1}{2}\left(\int_0^1x^{-1/2}(1-x)^{\lambda-1/2}dx\right)^2
=\frac{\beta^2(1/2,\lambda+1/2)}{ \Gamma^2(\lambda+1/2)} 
= \frac{\Gamma^2(1/2)}{\Gamma^2(\lambda+1)},
\end{align*}
where   $r_{\lambda}(h)$ is defined as \eqref{eq:proofth1:2:Ulambda2}.
The proof is completed. $\Box$

\section{Small ball probabilities}
\label{sect:th:2}
\setcounter{equation}{0}

In this section, we study  the   small ball probabilities of a weighted integrated fractional process.  In the sequel,  for two  function $f(\epsilon)$ and $g(\epsilon)$, we denote the notations $f(\epsilon)\preccurlyeq g(\epsilon)$ if $\limsup_{\epsilon\to 0}f(\epsilon)/g(\epsilon)$ is bounded,  
   $f(\epsilon)\prec g(\epsilon)$ if $\limsup_{\epsilon\to 0}f(\epsilon)/g(\epsilon)\le 1$, $f(\epsilon)\approx g(\epsilon)$ if both $f(\epsilon)\preccurlyeq g(\epsilon)$ and $g(\epsilon)\preccurlyeq f(\epsilon)$, $f(\epsilon)\sim g(\epsilon)$ if $\lim_{\epsilon\to 0}f(\epsilon)/g(\epsilon)=1$, and $f(\epsilon)\ll g(\epsilon)$ if $\limsup_{\epsilon\to 0}f(\epsilon)/g(\epsilon)\le 0$.

\subsection{Small ball probabilities under the $t^{-\alpha}$-weighted sup-norm}
To establish the small probabilities under $t^{-\alpha}$-weighted sup-norm given in \eqref{eq:th:1:1:1}, we first consider the special cases of $m=0$. 

\begin{proposition}\label{prop:1}
\begin{itemize}
  \item[(i)] For all $\gamma>0$ and $\alpha<H+\gamma$, 
\begin{align} \label{eq:lem:smallballofFB:1}
 \lim_{\epsilon\to 0} \epsilon^{1/(H+\gamma)} \log \pr\left(\sup_{0\le  t\le 1}t^{-\alpha}\left| I_{\gamma}(B_H)(t) \right|<a_H \epsilon\right) 
 =-\frac{\kappa_{H+\gamma}}{1-\alpha/(H+\gamma)}. 
\end{align}
  \item[(ii)]  Let  $\lambda>0$, and $W_{\lambda}(t)$ be defined as \eqref{eqdefW}.  Then for $\alpha<\lambda$, 
\begin{align}\label{eq:smallballofW}
\lim_{\epsilon\to 0}\epsilon^{1/\lambda} \log \pr\left(\sup_{0\le  t\le 1}\left|W_{\lambda}(t)\right|/t^{\alpha}< \epsilon\right)
=-\frac{\kappa_{\lambda}}{1-\alpha/\lambda}.
\end{align}
\end{itemize}
\end{proposition}
When $\alpha=0$,   \eqref{eq:smallballofW} is established  by Li and Linde \cite{LiLinde1998b},  and   \eqref{eq:lem:smallballofFB:1} is established by Li and Linde \cite{LiLinde1999}.  
   To   derive the general small ball probabilities given in    Proposition \ref{prop:1} and Theorem \ref{th:1:1}, we need some lemmas.

 \begin{lemma}\label{lem:WeakGaussCor} (Li \cite[Theorem 1.1]{LiWB1999})  Let $X$   be any centered  Gaussian random element in a separable Banach
space $E$.  Then
for any $0 <\lambda < 1$, any symmetric, convex sets $A$ and $B$ in $E$,
\begin{align}\label{eq:GaussCorWeak} \pr( X\in A\cap B)\ge \pr (X\in \lambda A) \pr (X\in (1-\lambda^2)^{1/2} B). 
\end{align}
\end{lemma}
\begin{lemma}\label{lem:ChungLIL1} (Li \cite[Theorem 1.2]{LiWB1999}) Let $X$ and $Y$ be any two centered joint Gaussian random elements  in a separable Banach
space with norm $\|\cdot\|$. If
$$\lim_{\epsilon\to 0}(\text{resp.} \; \liminf_{\epsilon\to 0},\; \limsup_{\epsilon\to 0})\epsilon^{\alpha} \log \pr(\|X\|<\epsilon)=-C_X  $$
and
 $$\lim_{\epsilon\to 0}\epsilon^{\alpha} \log \pr(\|Y\|<\epsilon)=0 $$
with $0< \alpha<\infty$ and $0<C_X<\infty$. Then
\begin{align}\label{eq:lem:ChungLIL1:1} & \lim_{\epsilon\to 0}(\text{resp.} \; \liminf_{\epsilon\to 0},\; \limsup_{\epsilon\to 0})\epsilon^{\alpha} \log \pr(\|X+Y\|<\epsilon)=-C_X, \\
\label{eq:lem:ChungLIL1:1}
&\lim_{\epsilon\to 0}(\text{resp.} \; \liminf_{\epsilon\to 0},\; \limsup_{\epsilon\to 0})\epsilon^{\alpha} \log \pr(\|X\|<\epsilon,\|Y\|<\epsilon)=-C_X.
\end{align}
\end{lemma}

\begin{lemma}\label{lem:kernel}  (Li and Linde \cite[Theorem 6.1] {LiLinde1999}) Let $\{Y(t);t\in [c,d]\}$ be a continuous, centered, Gaussian process with 
$$ -\log \pr\left(\sup_{t\in [c,d]}|Y(t)|<\epsilon\right)\preccurlyeq \;(resp. \ll)\; \epsilon^{-\alpha}\Big(\log \frac{1}{\epsilon}\Big)^{\nu}, $$
where $\alpha>0$ and $\nu\ge 0$. Assume that $K(t,s):[a,b]\times[c,d]\to \mathbb R$ is a kernel function which satisfies the H\"older inequality 
$$ \int_c^d \big|K(t^{\prime},s)-K(t^{\prime\prime},s)\big|ds\le C\big|t^{\prime}-t^{\prime\prime}\big|^{\gamma },\;\; t^{\prime},t^{\prime\prime}\in [a,b], $$
for some $\gamma \in (0,1]$ and $C>0$. Then
$$ -\log \pr\left(\sup_{t\in [a,b]}\big|\int_c^dK(t,s)Y(s)ds\big|<\epsilon\right)\preccurlyeq \;(resp. \ll)\; \epsilon^{-\alpha/(\alpha\gamma +1)}\Big(\log \frac{1}{\epsilon}\Big)^{\nu/(\alpha\gamma +1)}. $$
In particular, 
\begin{equation} \label{eq:lem:kernel:1} -\log \pr\left(\sup_{t\in [a,b]}\big|\int_a^tY(s)ds\big|<\epsilon\right)\preccurlyeq \;(resp. \ll)\; \epsilon^{-\alpha/(\alpha+1)}\Big(\log \frac{1}{\epsilon}\Big)^{\nu/(\alpha+1)}. 
\end{equation}
\end{lemma} 

One of the keys to prove Proposition \ref{prop:1} and \eqref{eq:th:1:1:1} is the following  property of a self-similar Gaussian process.   
\begin{lemma}\label{lem:selfsimilar}  Let $  X(t)$ be a continuous, centered, self-similar Gaussian process of index $\tau>0$.
Suppose that
\begin{equation}\label{eq:selfsimilar.1}-\log \pr\Big(\sup_{a\le t\le b} \left|  X(t)\right|<\epsilon\Big)\preccurlyeq \; (\text{resp.} \ll)\; \epsilon^{-\beta}\Big(\log\frac{1}{\epsilon}\Big)^{\nu },   
\end{equation} 
with $0<\beta<\infty$, $0\le \nu <\infty$, $0\le a<b<\infty$. Then for all $0\le c<d<\infty$,
\begin{equation}\label{eq:selfsimilar.2}-\log \pr\left(\sup_{c\le t\le d} \left | X(t)\right |<\epsilon\right)\preccurlyeq \; (\text{resp.} \ll)\; \epsilon^{-\beta}\Big(\log\frac{1}{\epsilon}\Big)^{\nu }.
\end{equation}

As corollaries, we have  
\begin{itemize}
\item[(i)] For all $\alpha<\tau$,
$$ -\log \pr\left(\sup_{0\le t\le 1} t^{-\alpha}\left | X(t)\right |<\epsilon\right)\preccurlyeq \; (\text{resp.} \ll)\; \epsilon^{-\beta}\Big(\log\frac{1}{\epsilon}\Big)^{\nu }; $$
  \item[(ii)] For all $\alpha<\tau+1$ and $\alpha_1<\tau+1$,  
\begin{equation}\label{eq:selfsimilar.3}-\log \pr\Big(\sup_{0\le t\le 1}|t^{-(\alpha-\alpha_1)}J_{\alpha_1}(X)(t)|< \epsilon\Big)
\preccurlyeq \; (\text{resp.} \ll)\;\epsilon^{- \beta/(\beta+1)}\Big(\log\frac{1}{ \epsilon}\Big)^{\nu /(\beta+1)}, 
\end{equation}
\begin{equation}\label{eq:selfsimilar.4}-\log \pr\left(\sup_{0\le t\le 1} \left |J_{\alpha}(X)(t)\right |<\epsilon\right)\preccurlyeq \; (\text{resp.} \ll)\; \epsilon^{-\beta/(\beta+1)}\Big(\log\frac{1}{\epsilon}\Big)^{\nu /(\beta+1)},  
\end{equation}
\begin{equation}\label{eq:selfsimilar.5}-\log \pr\left(\sup_{0\le t\le 1} \left|t^{-\alpha}I(X)(t) \right |<\epsilon\right)\preccurlyeq \; (\text{resp.} \ll)\; \epsilon^{-\beta/(\beta+1)}\Big(\log\frac{1}{\epsilon}\Big)^{\nu /(\beta+1)};   
\end{equation} 
\item[(iii)] 
For all $m\ge 1$, $\alpha_1+\cdots+\alpha_i<\tau+i$, $i=1,\ldots,m$, and $\alpha<\tau+m$,   
\begin{align} \label{eq:selfsimilar.6}
-\log \pr& \left(\sup_{0\le t\le 1} \left|t^{-(\alpha-\alpha_1-\cdots-\alpha_m)} J_{m,\bm\alpha}(X)(t)-t^{-\alpha}  I_{m}(X)(t) \right|<\epsilon\right)\nonumber\\
&\preccurlyeq \; (\text{resp.} \ll)\; \epsilon^{-\beta/((m+1)\beta+1)}\Big(\log\frac{1}{\epsilon}\Big)^{\nu /((m+1)\beta+1)}.
\end{align} 
\end{itemize}
\end{lemma} 
{\bf Proof.}  For \eqref{eq:selfsimilar.2}, notice that $\sup_{t\in[c,d]}|X(t)|\le  \sup_{t\in[0,d]}|X(t)|\overset{d}=d^{\tau}\sup_{t\in[0,1]}|X(t)|$. Without loss of generality, we assume $[c,d]=[0,1]$.  By \eqref{eq:selfsimilar.1} and the self-similarity of $X(t)$, 
$$ -\log \pr\Big(\sup_{a/b\le t\le 1} \left|  X(t)\right|<\epsilon\Big)=-\log \pr\Big(\sup_{a\le t\le b} \left|  X(t)\right|<\epsilon b^{\tau}\Big)\preccurlyeq \epsilon^{-\beta}\Big(\log\frac{1}{\epsilon}\Big)^{\nu }. $$
Thus, there exist  $0\le \delta<1$  and two  positive constants  $\kappa$ and $\epsilon_0$ such that
 \begin{equation}\label{eq:prooflem:ChungLIL5.1}
  \log \pr\Big(\sup_{\delta\le t\le 1} |X(t)|< \epsilon\Big) 
\ge
 -\kappa\epsilon^{-\beta}\Big(\log\frac{1}{\epsilon}\Big)^{\nu }, \;\; 0<\epsilon\le \epsilon_0.
 \end{equation}
 If $\delta=0$, then \eqref{eq:selfsimilar.2} holds. Suppose $0<\delta<1$. Choose $0<\lambda<1$ such that  $\lambda_0=\lambda\delta^{-\tau}>1$. By  Lemma \ref{lem:WeakGaussCor}, 
 \begin{align*}
   \pr\Big(\sup_{0\le t \le 1} \left|X(t)\right|<\epsilon\Big)
\ge  & \pr\Big(\sup_{0\le t\le \delta} \left|X(t)\right|< \lambda\epsilon\Big)\pr\Big( \sup_{\delta \le t\le 1} \left|X(t)\right|< \sqrt{1-\lambda^2}\epsilon \Big)\\
= &   \pr\Big(\sup_{0\le t\le 1}\left|X(t)\right|<\lambda_0\epsilon\Big)\pr\Big( \sup_{\delta\le t\le 1} \left|X(t)\right|< \sqrt{1-\lambda^2}\epsilon \Big).
 \end{align*}
 For $0<\epsilon\le \epsilon_0/\lambda_0^3$, choose $k$ such that $\lambda_0^k\epsilon\le \epsilon_0<\lambda_0^{k+1}\epsilon$. Then, we have that
 \begin{align*}
&\log\pr\Big(\sup_{0\le t \le 1} \left|X(t)\right|<\epsilon\Big) \\   
\ge  & \log \pr\Big(\sup_{0\le t \le 1} \left|X(t)\right|<\lambda_0^{k+1}\epsilon\Big)+\sum_{i=0}^{k}\log \pr\Big( \sup_{\delta\le t\le 1} \left|X(t)\right|< \sqrt{1-\lambda^2}\lambda_0^i\epsilon \Big)\\
\ge & \log \pr\Big(\sup_{0\le t \le 1} \left|X(t)\right|<\epsilon_0\Big)- \kappa 
\sum_{i=0}^k \big(\sqrt{1-\lambda^2}\lambda_0^i\epsilon\big)^{-\beta}\Big(\log\frac{1}{\sqrt{1-\lambda^2}\lambda_0^i\epsilon}\Big)^{\nu }\\
\ge  &  \log\pr\Big(\sup_{0\le t \le 1} \left|X(t)\right|<\epsilon_0\Big)-\frac{\kappa}{(1-\lambda^2)^{\beta/2}(1-\lambda_0^{-\beta})} \epsilon^{-\beta}\Big(\log\frac{1}{\sqrt{1-\lambda^2} \epsilon}\Big)^{\nu }.
\end{align*}  
When $\preccurlyeq$ is replaced $\ll$, $\kappa$ can be an arbitrarily small constant. \eqref{eq:selfsimilar.2}   is proven.

 For (i), we let $Y(t)=t^{-\alpha}X(t)$. It is sufficient to notice that $Y(t)$ is a continuous,   centered, self-similar Gaussian process of index $\tau-\alpha$ with
  $$ -\log \pr\left(\sup_{\delta\le t\le 1}|Y(t)|<\epsilon\right)\le -\log \pr\left(\sup_{\delta\le t\le 1}|X(t)|<\epsilon \delta^{-|\alpha|}\right)\preccurlyeq \; (\text{resp.} \ll)\; \epsilon^{-\beta}\Big(\log\frac{1}{\epsilon}\Big)^{\nu }. $$
  
  For (ii), we let $Y(t)=t^{-(\alpha-\alpha_1)}\int_0^t s^{-\alpha_1} X(s)ds$, $\alpha,\alpha_1<\tau+1$. Then $Y(t)$ is well-defined and a continuous,   centered, self-similar Gaussian process of index $\tau+1-\alpha$. 
For $0<\delta<1$, let 
$ K(t,s)=t^{-(\alpha-\alpha_1)}s^{-\alpha_1}I\{\delta\le s\le t\}. $
Then $K(t,s)$ satisfies the H\"older condition
$$ \int_{0}^1\left|K(t^{\prime},s)- K(t^{\prime\prime},s)\right|ds\le c|t^{\prime}-t^{\prime\prime}|, \;\; t^{\prime}, t^{\prime\prime}\in [\delta,1], $$
and $\sup_{\delta\le t\le 1}|Y(t)|=\sup_{\delta\le t\le 1}\big|\int_{0}^1 K(t,s) X(s)ds+t^{-(\alpha-\alpha_1)}\int_0^{\delta} s^{-\alpha_1}X(s)ds\big|$.
By \eqref{eq:selfsimilar.2}  and Lemma \ref{lem:kernel}, we have that
\begin{equation}\label{eq:lemChungLIL5.ad.1}
 -\log \pr\Big(\sup_{\delta\le t\le 1}\big|\int_{0}^1 K(t,s) X(s)ds\big|< \epsilon\Big)\preccurlyeq \epsilon^{- \beta/(\beta+1)}\Big(\log\frac{1}{ \epsilon}\Big)^{\nu /(\beta+1)}. 
 \end{equation}
On the other hand, since $\int_0^{\delta} s^{-\alpha_1}X(s)ds$ is a centered normal random variable, it is easily checked that 
$$ \pr\Big(\sup_{\delta\le t\le 1} t^{-(\alpha-\alpha_1)}\big|\int_0^{\delta} s^{-\alpha_1}X(s)ds\big|<  \epsilon\Big) =\pr\Big((1\vee \delta^{\alpha_1-\alpha})\big|\int_0^{\delta} s^{-\alpha_1}X(s)ds\big|<  \epsilon\Big)\approx \epsilon. $$
By applying Lemma \ref{lem:WeakGaussCor}, it follows that 
$$-\log \pr\Big(\sup_{\delta\le t\le 1}|Y(t)|< \epsilon\Big)
\preccurlyeq \epsilon^{- \beta/(\beta+1)}\Big(\log\frac{1}{ \epsilon}\Big)^{\nu /(\beta+1)}. 
$$  
  By \eqref{eq:selfsimilar.2}, \eqref{eq:selfsimilar.3} is proven.   \eqref{eq:selfsimilar.4} is a special case of \eqref{eq:selfsimilar.3} with $\alpha_1=\alpha$,  and  
\eqref{eq:selfsimilar.5} is a special case of \eqref{eq:selfsimilar.3} with $\alpha_1=0$.   

For (iii),  we let $Y(t)=I(X)(t)$.  Then
$$ -\log\pr\left(\sup_{0\le t\le 1}|Y(t)|<\epsilon\right) \preccurlyeq  \epsilon^{-\beta/(\beta+1)}\Big(\log\frac{1}{\epsilon}\Big)^{\nu /(\beta+1)} $$
as shown. Notice that 
\begin{align*}
& t^{-(\alpha-\alpha_1)}J_{\alpha_1}(X)(t)-t^{-\alpha}I(X)(t)
 =  t^{-(\alpha-\alpha_1)}\int_0^t(s^{-\alpha_1}-t^{-\alpha_1})d Y(s)\\
  =&\alpha_1t^{-\{(\alpha+1)-(\alpha_1+1)\}}\int_0^t s^{-(\alpha_1+1)}Y(s) ds, 
\end{align*}
 and $Y(t)$ is a continuous, centered, self-similar Gaussian process of index $\tau+1$.  By \eqref{eq:selfsimilar.3}, we have that 
\begin{align}\label{eq:prooflem:ChungLIL5.2}
  &-\log \pr \left(\sup_{0\le t\le 1} \left |t^{-(\alpha-\alpha_1)}J_{\alpha_1}(X)(t)-t^{-\alpha}I(X)(t)\right |<\epsilon\right)\nonumber\\
\preccurlyeq & \epsilon^{-\frac{\beta/(\beta+1)}{\beta/(\beta+1)+1}}\Big(\log\frac{1}{\epsilon}\Big)^{\frac{\nu /(\beta+1)}{\beta/(\beta+1)+1}}
=\epsilon^{-\beta/(2\beta+1)}\Big(\log\frac{1}{\epsilon}\Big)^{\nu /(2\beta+1)}
\end{align}
for all $\alpha,\alpha_1<\tau+1$. Thus, \eqref{eq:selfsimilar.6} holds for $m=1$.

For $m\ge 2$, we let $X(t)=  J_{m-1,\bm\alpha_{m-1}}(X)(t)-t^{-\alpha_1-\cdots-\alpha_{m-1}}  I_{m-1}(X)(t)$, where $\bm\alpha_{m-1}=(\alpha_1,\ldots,\alpha_{m-1})$. 
Suppose that \eqref{eq:selfsimilar.6} holds for $m-1$. Then
$$-\log \pr\left(\sup_{0\le t\le 1} \left|X(t) \right|<\epsilon\right)\preccurlyeq \epsilon^{-\beta/(m\beta+1)}\Big(\log\frac{1}{\epsilon}\Big)^{\nu /(m\beta+1)}. $$
It is easily seen that $X(t)$ is a continuous, centered, self-similar Gaussian process of index $\tau^{\prime}=\tau+m-1-\alpha_1-\cdots-\alpha_{m-1}$. 
By \eqref{eq:selfsimilar.3} of Lemma \ref{lem:selfsimilar},
\begin{align*}
 & -\log \pr  \left(\sup_{0\le t\le 1} t^{-(\beta-\alpha_m)}\left|\int_0^t s^{-\alpha_m} X(s)ds \right|<\epsilon\right)\\
  &\preccurlyeq \epsilon^{-\frac{\beta/(m\beta+1)}{\beta/(m\beta+1)+1}}\Big(\log\frac{1}{\epsilon}\Big)^{\frac{\nu /(m\beta+1)}{\beta/(m\beta+1)+1}}=\epsilon^{-\beta/((m+1)\beta+1)}\Big(\log\frac{1}{\epsilon}\Big)^{\nu /((m+1)\beta+1)}
\end{align*}
holds for all $\alpha_m<\tau^{\prime}+1$ and $\beta<\tau^{\prime}+1$. Choosing $\beta=\alpha-\alpha_1-\cdots-\alpha_{m-1}$ 
 yields
\begin{align}  \label{eq:prooflem:ChungLIL5.3}
  -\log \pr& \left(\sup_{0\le t\le 1} t^{-(\alpha-\alpha_1-\cdots-\alpha_m)}\left|J_{m,\bm\alpha}(X)(s)ds-J_{\alpha_1+\cdots+\alpha_m}\big(I_{m-1}X\big)(t)\right|<\epsilon\right)
\nonumber\\
&\preccurlyeq \epsilon^{-\beta/((m+1)\beta+1)}\Big(\log\frac{1}{\epsilon}\Big)^{\nu /((m+1)\beta+1)}.
\end{align}
On the other hand, by repeating \eqref{eq:selfsimilar.5}  with $\alpha=0$, we have that 
$$ -\log \pr  \left(\sup_{0\le t\le 1} \left |I_{m-1}X(t)\right |<\epsilon\right) \preccurlyeq  \epsilon^{-\beta/((m-1)\beta+1)}\Big(\log\frac{1}{\epsilon}\Big)^{\nu /((m-1)\beta+1)}.
$$
Notice that $I_{m-1}X$ is a centered, self-similar Gaussian process of index $m-1+\tau$. \eqref{eq:prooflem:ChungLIL5.2} applies to $I_{m-1}X$ and gives that
\begin{align}  \label{eq:prooflem:ChungLIL5.4}
-\log \pr& \left(\sup_{0\le t\le 1} \left |t^{-(\alpha-\alpha_1-\cdots-\alpha_m)}J_{\alpha_1+\cdots+\alpha_m}\big(I_{m-1}X\big)(t)-t^{-\alpha}I_{m}(X)(t)\right |<\epsilon\right)\nonumber\\
&\preccurlyeq  \epsilon^{-\beta/((m+1)\beta+1)}\Big(\log\frac{1}{\epsilon}\Big)^{\nu /((m+1)\beta+1)}, \;\; \alpha, \alpha_1+\cdots+\alpha_m<m+\tau. 
\end{align}
By Lemma \ref{lem:WeakGaussCor},   \eqref{eq:selfsimilar.6} follows from \eqref{eq:prooflem:ChungLIL5.3} and \eqref{eq:prooflem:ChungLIL5.4}.

Finally, if $\preccurlyeq$ in \eqref{eq:selfsimilar.1} is replaced by $\ll$, then \eqref{eq:lemChungLIL5.ad.1} holds with $\ll$ taking the place of $\preccurlyeq$. 
Thus,  \eqref{eq:selfsimilar.3}-\eqref{eq:selfsimilar.6} hold.  The proof of the lemma is completed. 
$\Box$

\begin{lemma}\label{lem:6} Let $B(t)$ be a standard Brownian motion. For every $\lambda>0$ and $0\le a<b$, we have that 
\begin{equation}\label{eq:lem:6:1} \lim_{\epsilon\to 0}\epsilon^{1/\lambda}\log\pr\left(\sup_{a\le t\le b}\frac{1}{\Gamma\big(\lambda+1/2\big)}\left|\int_{0}^{a}  (t-s)^{\lambda-1/2} d B(s)\right|<\epsilon\right)=0.
\end{equation}
\end{lemma}
{\bf Proof}. Without loss of generality, we assume $b=1$, $a=\delta\in [0,1)$. At first, we assume $0<\lambda<1$.  For $0\le \delta<1$, let 
$$Z_{\delta,\lambda}(t)= \int_{-\infty}^{\delta}\left((t-s)^{\lambda-1/2}-(-s)_+^{\lambda-1/2}\right)dB(s), t\ge 0,$$
By  Lemma 4.3 of Lifshits and Linde \cite{LifshitsLinde2005}, 
$$ \lim_{\epsilon\to 0}\epsilon^{\rho}\log \pr\left(\sup_{\delta\le t\le 1}|Z_{\delta,\lambda}|<\epsilon\right)=0, \; \forall \rho>0. $$
In particular,
 $$ \lim_{\epsilon\to 0}\epsilon^{\rho}\log \pr\left(\sup_{\delta\le t\le 1}|Z_{0,\lambda}|<\epsilon\right)
 \ge \lim_{\epsilon\to 0}\epsilon^{\rho}\log \pr\left(\sup_{0\le t\le 1}|Z_{0,\lambda}|<\epsilon\right)=0, \; \forall \rho>0. $$
 Notice that
 $$ \int_{0}^{\delta}  (t-s)^{\lambda-1/2} d B(s)=Z_{\delta,\lambda}-Z_{0,\lambda}. $$
 By Lemma \ref{lem:ChungLIL1},
 $$  \lim_{\epsilon\to 0}\epsilon^{\rho}\log\pr\left(\sup_{\delta\le t\le 1}\left|\int_{0}^{\delta}  (t-s)^{\lambda-1/2} d B(s)\right|<\epsilon\right)=0, \;\forall \rho>0. $$
 \eqref{eq:lem:6:1} holds. 
 
 Next, assume $\lambda=1$. Notice that 
 $$\Big\{\int_0^{\delta}(t-s)^{1/2}d B(s);t\ge \delta\Big\}\overset{d}=\Big\{\int_0^{\delta}(t-\delta+s)^{1/2}d B(s);t\ge \delta\Big\},$$ 
 $$ \int_0^{\delta}(t-\delta+s)^{1/2}d B(s)=t^{1/2}B(\delta)-\frac{1}{2}\int_0^{\delta} (t-\delta+s)^{-1/2}B(s)ds, $$
 $$ \pr\left(\sup_{\delta\le t\le 1}|t^{1/2}B(\delta)|<\epsilon\right)=
 \pr\left(|B(\delta)|<\epsilon\right)\approx \epsilon. $$
 By Lemma \ref{lem:ChungLIL1}, it is sufficient to show that
\begin{equation}\label{eq:lem:6:2} \liminf_{\epsilon\to 0} \epsilon\log \pr\left(\sup_{0\le t\le 1-\delta}\left|\int_0^{\delta} (t+s)^{-1/2}B(s)ds\right|<\epsilon\right)=0.
\end{equation}
 Let $0\le a<b, 0\le c<d$, $K(t,s)=(t+s)^{-1/2}, t\in[a,b],s\in[c,d]$. When $a>0$ or $c>0$, it is easily checked that
 $$\int_a^b \big|K(t^{\prime},s)-K(t^{\prime\prime},s)\big|ds\le C|t^{\prime}-t^{\prime\prime}|,\;\; t^{\prime},t^{\prime\prime}\in [a,b]. $$
  When $a=c=0$, assume $0\le t^{\prime}<t^{\prime\prime}\le b$. Then
 \begin{align*}
 \int_a^b \big|K(t^{\prime},s)-K(t^{\prime\prime},s)\big|ds& 
  =2\left[ (t^{\prime}+s)^{1/2}-(t^{\prime\prime}+s)^{1/2}\right]\Big|_{s=0}^b \\
  \le &\frac{2|t^{\prime}-t^{\prime\prime}|}{(t^{\prime})^{1/2}+(t^{\prime\prime})^{1/2}} 
  \le 2|t^{\prime}-t^{\prime\prime}|^{1/2}. 
  \end{align*}
 By Lemma \ref{lem:kernel}, 
\begin{align} \label{eq:lem:6:3} & -\log \pr\left(\sup_{a\le t\le b}\big|\int_c^d(t+s)^{-1/2} B(s)ds\big|<\epsilon\right) \preccurlyeq \epsilon^{-\frac{2}{2*1+1}}=\epsilon^{-2/3}, \text{ if } a+c>0, \\
& -\log \pr\left(\sup_{a\le t\le b}\big|\int_c^d(t+s)^{-1/2} B(s)ds\big|<\epsilon\right) \preccurlyeq \epsilon^{-\frac{2}{2*\frac{1}{2}+1}}=\epsilon^{-1}, \text{ if } a=c=0. \label{eq:lem:6:4} 
\end{align}
  By \eqref{eq:lem:6:4}, the limit value in \eqref{eq:lem:6:2} is finite. We denote it by $\kappa$. By \eqref{eq:lem:6:3} and Lemma \ref{lem:ChungLIL1},
 \begin{align*}
 \kappa=&\liminf_{\epsilon\to 0} \epsilon\log \pr\left(\sup_{0\le t\le \delta}\left|\int_0^{\delta} (t+s)^{-1/2}B(s)ds\right|<\epsilon\right)\\
 =& \liminf_{\epsilon\to 0} \epsilon\log \pr\left(\sup_{0\le t\le \delta}\left|\int_0^1 (t+s)^{-1/2}B(s)ds\right|<\epsilon\right)\\
 = &
 \liminf_{\epsilon\to 0} \epsilon\log \pr\left(\sup_{0\le t\le 1}\left|\int_0^1 (t+s)^{-1/2}B(s)ds\right|<\epsilon\right).
 \end{align*} 
 On the other hand,
\begin{align*}
\sup_{0\le t\le \delta}\left|\int_0^{\delta} (t+s)^{-1/2}B(s)ds\right|=&
 \sup_{0\le t\le 1}\left|\int_0^1 (\delta t+\delta s)^{-1/2}B(\delta s)\delta ds\right|\\
 \overset{\mathcal{D}}=& \delta\sup_{0\le t\le 1}\left|\int_0^1 (  t+  s)^{-1/2}B(  s)  ds\right|.
 \end{align*}
 Thus, $\kappa=\delta\kappa$. Hence, we must have $\kappa=0$. \eqref{eq:lem:6:2} is proven. So, \eqref{eq:lem:6:1} holds for all $0<\lambda\le 1$.  
 
 Finally, we show \eqref{eq:lem:6:1} for all $\lambda\ge 1$ by the induction. Suppose \eqref{eq:lem:6:1} holds for $\lambda>0$. Let $Y_{\lambda}(t)=\frac{1}{\Gamma\big(\lambda+1/2\big)}\int_{0}^{\delta}  (t-s)^{\lambda-1/2} d B(s)$. 
 By applying Lemma \ref{lem:kernel} to   $\alpha=1/\lambda$ and $\nu=0$ (c.f. \eqref{eq:lem:kernel:1}), we have
 $$ \lim_{\epsilon\to 0}\epsilon^{\frac{1/\lambda}{1/\lambda+1}}\log \pr\left(\sup_{\delta\le t\le 1}\big|\int_{\delta}^t Y_{\lambda}(s)ds\big|<\epsilon\right)=0.$$
 Notice that
 $$ \int_{\delta}^t Y_{\lambda}(s)ds =Y_{\lambda+1}(t)-\frac{1}{\Gamma\big(\lambda+3/2\big)}\int_0^{\delta}(\delta-s)^{\lambda+1/2}d B(s). $$ 
 The second term above is a centered normal random variables.  By Lemma \ref{lem:ChungLIL1} again, 
 $$ \lim_{\epsilon\to 0}\epsilon^{1/(\lambda+1)}\log \pr\left(\sup_{\delta\le t\le 1}|Y_{\lambda+1}(t)|<\epsilon\right)=0.$$
 By the induction, \eqref{eq:lem:6:1} holds for all $\lambda>0$.  $\Box$.

 {\bf Proof of Proposition \ref{prop:1}.} Notice that
 $ B_H(t)=a_H\big(W_H(t)+Z_H(t)\big) $, 
 where 
$a_H$, $W_H(t)$ and $Z_H(t)$   are defined as \eqref{eqdefaH}, \eqref{eqdefW} and  \eqref{eqdefZH}, respectively.
  We first consider $W_{\lambda}$. By Theorem 2.1 of Li and Linde \cite{LiLinde1998b},   \eqref{eq:smallballofW} holds for $\alpha=0$, c.f. \eqref{eq:constant:1}.
Write $P(\epsilon)=\log \pr\left(\sup_{0\le  t\le 1}\frac{1}{t^{\alpha}} \left|W_{\lambda}(t)\right|< \epsilon\right)$. Notice that
$W_{\lambda}(t)$ is a centered, self-similar Gaussian process of index $\lambda$.  By Lemma \ref{lem:selfsimilar} (i),  we have that
$0\le -P(\epsilon)\preccurlyeq \epsilon^{-1/\lambda}. $
Thus, $\limsup_{\epsilon\to 0}\epsilon^{1/\lambda}P(\epsilon)$ and $\liminf_{\epsilon\to 0}\epsilon^{1/\lambda}P(\epsilon)$ are finite. 
Now, for $0<\delta<1$ we have that
\begin{align*}
&\pr\left(\sup_{0\le  t\le 1}\frac{1}{t^{\alpha}}\Big| W_{\lambda}(t) \Big|< \epsilon\right) 
\ge  \pr\left(\sup_{0\le  t\le \delta}\frac{1}{t^{\alpha}}\Big| W_{\lambda}(t) \Big|< \epsilon, 
\sup_{\delta\le t\le 1}\Big| W_{\lambda}(t) \Big|<  \delta^{|\alpha|}\epsilon\right)\\  
= &\pr\left(\sup_{0\le  t\le \delta}\frac{1}{t^{\alpha}}\Big| W_{\lambda}(t) \Big|< \epsilon, \right. \\
&\qquad \left.\sup_{\delta\le t\le 1}\left|\int_{\delta}^t \frac{(t-s)^{\lambda-1/2}}{\Gamma\big(\lambda+1/2\big)}d B(s)+\int_{0}^{\delta} \frac{(t-s)^{\lambda-1/2}}{\Gamma\big(\lambda+1/2\big)}d B(s)\right|<\delta^{|\alpha|}\epsilon\right).
\end{align*}
By Lemma \ref{lem:6}, 
$$ \lim_{\epsilon\to 0}\epsilon^{1/\lambda}\log\pr\left(\sup_{\delta\le t\le 1}\left|\int_{0}^{\delta} \frac{(t-s)^{\lambda-1/2}}{\Gamma\big(\lambda+1/2\big)}d B(s)\right|<\epsilon\right)=0. $$
It is easily checked that $\sup_{0\le t\le \delta}t^{-\alpha}|W_{\lambda}(t)|\overset{\mathcal{D}}=\delta^{\lambda-\alpha}\sup_{0\le t\le 1}t^{-\alpha}|W_{\lambda}(t)|$ and 
$$ \sup_{\delta\le t\le 1}\left|\int_{\delta}^t \frac{(t-s)^{\lambda-1/2}}{\Gamma\big(\lambda+1/2\big)}d B(s)\right|
\overset{\mathcal{D}}= \sup_{0\le t\le 1-\delta}\left|W_{\lambda}(t)\right|\overset{\mathcal{D}}=
(1-\delta)^{\lambda} \sup_{0\le t\le 1}\left|W_{\lambda}(t)\right|. $$

By Lemma \ref{lem:ChungLIL1}, \eqref{eq:constant:1} and the independence,
\begin{align*}
&\liminf_{\epsilon\to 0}\epsilon^{1/\lambda}P(\epsilon)\\
\ge &\liminf_{\epsilon\to 0}\epsilon^{1/\lambda}\log \pr\left(\sup_{0\le  t\le \delta}\frac{1}{t^{\alpha}}\Big| W_{\lambda}(t) \Big|< \epsilon, \sup_{\delta\le t\le 1}\left|\int_{\delta}^t \frac{(t-s)^{\lambda-1/2}}{\Gamma\big(\lambda+1/2\big)}d B(s)\right|<\delta^{|\alpha|}\epsilon\right)\\
= &\liminf_{\epsilon\to 0}\epsilon^{1/\lambda}\log\left\{\pr\left(\sup_{0\le  t\le \delta}\frac{1}{t^{\alpha}}\Big| W_{\lambda}(t) \Big|< \epsilon\right) \pr\left(\sup_{\delta\le t\le 1}\left|\int_{\delta}^t \frac{(t-s)^{\lambda-1/2}}{\Gamma\big(\lambda+1/2\big)}d B(s)\right|<\delta^{|\alpha|}\epsilon\right)\right\}\\
=& \liminf_{\epsilon\to 0}\epsilon^{1/\lambda}P(\delta^{\alpha-\lambda}\epsilon)
+\lim_{\epsilon\to 0}\epsilon^{1/\lambda} \log\pr\left( \sup_{0\le t\le 1}\Big|W_{\lambda}(t) \Big|< \frac{\delta^{|\alpha|}\epsilon}{(1-\delta)^{\lambda}}\right)  \\
=& \delta^{1-\alpha/\lambda}\liminf_{\epsilon\to 0}\epsilon^{1/\lambda}P(\epsilon)
-\kappa_{\lambda}\frac{1-\delta}{\delta^{|\alpha|/\lambda}}.
\end{align*} 
 We conclude that
$$ \liminf_{\epsilon\to 0}\epsilon^{2/\lambda}P(\epsilon)\ge -\kappa_{\lambda}\frac{1-\delta}{\delta^{|\alpha|/\lambda}(1-\delta^{1-\alpha/\lambda})}\to -\frac{\kappa_{\lambda}}{1-\alpha/\lambda}
\text{
as }  \delta\to 1. $$ 

On the other hand, 
for $0<\delta<1$,  
\begin{align*}
&\pr\left(\sup_{0\le  t\le 1}\frac{1}{t^{\alpha}}\Big| W_{\lambda}(t) \Big|< \epsilon\right) 
\le  \pr\left(\sup_{0\le  t\le \delta}\frac{1}{t^{\alpha}}\Big| W_{\lambda}(t) \Big|< \epsilon, 
\sup_{\delta\le t\le 1}\Big| W_{\lambda}(t) \Big|<  \delta^{-|\alpha|}\epsilon\right).
\end{align*}
With the same argument above, we have that 
\begin{align*}
& \limsup_{\epsilon\to 0}\epsilon^{1/\lambda}P(\epsilon)\\
\le & \limsup_{\epsilon\to 0}\epsilon^{1/\lambda}P(\delta^{\alpha-\lambda}\epsilon)+
\lim_{\epsilon\to 0}\epsilon^{1/\lambda} \log\pr\left( \sup_{0\le t\le 1}\Big|W_{\lambda}(t)\Big|< \frac{\delta^{-|\alpha|}\epsilon}{(1-\delta)^{\lambda}}\right)\\
=& \delta^{1-\alpha/\lambda}\limsup_{\epsilon\to 0}\epsilon^{1/\lambda}P(\epsilon)- \kappa_{\lambda}(1-\delta)\delta^{|\alpha|/\lambda}.
\end{align*} 
Thus
$$ \limsup_{\epsilon\to 0}\epsilon^{1/\lambda}P(\epsilon)\le -\kappa_{\lambda}
\frac{(1-\delta)\delta^{|\alpha|/\lambda}}{1-\delta^{1-\alpha/\lambda}} \to -\frac{\kappa_{\lambda}}{1-\alpha/\lambda}
\;\text{
as } \; \delta\to 1. $$ 
\eqref{eq:smallballofW} is proven.  

Next, we consider $I_{\gamma}(B_H)$.   Let  $0<H<1$, and $Z_H(t)$ be defined as \eqref{eqdefZH}. By Lemma 4.3 of Lifshits and Linde \cite{LifshitsLinde2005}, 
 $$ -\log \pr\left(\sup_{0\le t\le 1}\left| Z_H(t)\right|< \epsilon\right)
\preccurlyeq   \epsilon^{-\beta}, \forall \beta>0. $$
When $\gamma>0$, let $K(t,s)=t^{-\alpha}(t-s)^{\gamma-1}I\{0\le s\le t\}$.  Then, for $0<\delta<1$,
$$ \int_0^1\left|K(t^{\prime\prime},s)-K(t^{\prime},s)\right| ds \le \begin{cases} c_{\alpha,\delta,\gamma}|t^{\prime\prime}-t^{\prime}|, & \gamma\ge 1,\\
c_{\alpha,\delta,\gamma}|t^{\prime\prime}-t^{\prime}|^{\gamma}, & 0<\gamma\le 1,
\end{cases}\;\; t^{\prime},t^{\prime\prime}\in [\delta,1].$$
Thus, by Lemma \ref{lem:kernel},
\begin{align*}
&  -\log \pr\left(\sup_{\delta\le t\le 1}t^{-\alpha}\left|\int_0^t(t-s)^{\gamma-1} Z_H(s)ds\right|< \epsilon\right)  
\preccurlyeq   \epsilon^{-\frac{\beta}{\beta*(\gamma\wedge 1)+1}}, \forall \beta>0. 
\end{align*}
  On the other hand, $t^{-\alpha}\int_0^t(t-s)^{\gamma-1} Z_H(s)ds$ is a continuous,  centered, self-similar Gaussian process of index $H+\gamma-\alpha$. By  \eqref{eq:selfsimilar.2} of Lemma \ref{lem:selfsimilar},
 \begin{equation}\label{eq:smallbofZ} -\log \pr\left(\sup_{0\le t\le 1}t^{-\alpha}\left|\int_0^t(t-s)^{\gamma-1} Z_H(s)ds\right|< \epsilon\right) 
\preccurlyeq   \epsilon^{-\beta}, \forall \beta>0 \text{ and } \alpha<H+\gamma. 
\end{equation} 
Notice that 
 $$ I_{\gamma}B_H(t)=a_H\big[I_{\gamma}(W_H)(t)+I_{\gamma}(Z_H)(t)\big], \; t\ge 0. $$
 and 
 $$ I_{\gamma}(W_H)(t)dt=W_{H+\gamma}(t). $$
 By     Lemma  \ref{lem:ChungLIL1} and \eqref{eq:smallballofW}, \eqref{eq:lem:smallballofFB:1} holds. $\Box$
 
 \bigskip
 Now, we give the proof of the small ball probabilities given by \eqref{eq:th:1:1:1}.

{\bf Proof of \eqref{eq:th:1:1:1}.}   \eqref{eq:th:1:1:1}    holds  for $\bm\alpha=\bm 0$ by Proposition \ref{prop:1}. By Lemma \ref{lem:ChungLIL1},  it is sufficient to show   that
\begin{align}\label{eq:proof:th:2:1}
-\log \pr& \left(\sup_{0\le t\le 1} \left|t^{-(\alpha-\alpha_1-\cdots-\alpha_m)} J_{m,\bm\alpha}(X)(t)-t^{-\alpha}  I_{m}(X)(t) \right|<\epsilon\right)\nonumber\\
&\preccurlyeq \epsilon^{-1/(H+\gamma+m+1)},\;\;\text{ for }  X=I_{\gamma}B_H \text{ or } W_{H+\gamma},
\end{align} 
for all $\gamma\ge 0$, $\alpha_1+\cdots+\alpha_i<H+\gamma+i$, $i=1,\ldots,m$, and $\alpha<H+\gamma+m$.

Let $X=I_{\gamma}B_H$ or $W_{H+\gamma}$. Notice that $X$   is a continuous, centered, self-similar Gaussian process of index $\tau=H+\gamma$  and
$$ -\log\pr\left(\sup_{0\le t\le 1}|X(t)|<\epsilon\right) \preccurlyeq \epsilon^{-1/(H+\gamma)}, $$
by   Proposition \ref{prop:1}.  Lemma \ref{lem:selfsimilar} (iii) applies to $X$  and gives \eqref{eq:proof:th:2:1}.   
The proof is completed. 
 $\Box$ 
  
\subsection{Small ball probabilities under the weighted $L^q$-norm} 
The $\sup$-norm in the equation \eqref{eq:th:1:1:1} can be replaced by the $L^q$-norm.
Let $0< q\le \infty$. For a function $w(t)$ defined  on the interval $(0,\infty)$ and a subinterval $I\subset (0,\infty)$,   denote the norms as follows: $\|x\|_I=\sup_{t\in I}|x(t)|$, 
$$ \|w\|_{L^q(I)}=\left(\int_I |w(t)|^qdt\right)^{1/q}  \text{if } 0< q<\infty \text{ and } \|w\|_{L^q(I)}=\underset{t\in I}{\text{ess sup}} |w(t)| \text{ if } q= \infty. $$
When $\sup$-norm  is replaced by  the $L^q$-norm, $q\ge 1$,  the constant $\kappa_{\lambda}$ in \eqref{eq:constant:1} should be  substituted with:
\begin{align}\label{eq:constant:2} \kappa(\lambda,q)=-\lim_{\epsilon\to 0}\epsilon^{1/\lambda}\log \pr\left(\|W_{\lambda}\|_{L^q[0,1]}<\epsilon\right), \;\; \kappa(\lambda,\infty)=\kappa_{\lambda}.  
\end{align} 

Lifshits and Linde \cite{LifshitsLinde2005} obtained  the following small ball probability for a  weighted fractional Browning motion with a general weight $w(t)$: 
\begin{align}\label{eq:Lifshits:1} \lim_{\epsilon\to 0}\epsilon^{1/H}\log \pr\left(\|wB_H\|_{L^q(0,\infty)}<a_H\epsilon\right)=-\kappa(H,q)\|w\|_{L^r(0,\infty)}^{1/H}, 
\end{align}
provided  $q\ge 1$, $\|w\|_{r,H,q}<\infty$ and $\frac{1}{r}=H+\frac{1}{q}$.  When $q=\infty$, $w$ in right hand   of \eqref{eq:Lifshits:1} has to be replaced by its regularization $w^{\ast}$ defined by  
\begin{align}\label{eq:th:3.1:2} w^{\ast}(s)=\lim_{\delta\to 0}  \underset{\{x:|x-s|<\delta\}}{\text{ess sup}}  |w(x)|.
\end{align}
Here and in the sequel,   $\|w\|_{r,\tau,q}$  is defined as 
\begin{align}\label{eq:normofweight} \|w\|_{r,\tau,q}=\left(\sum_{k=-\infty}^{\infty}2^{kr\tau} \|w\|^r_{L^q(2^{k-1},2^k]} \right)^{1/r}, \;\; 1\le q\le \infty, 0<r,\tau<\infty.
\end{align}
 This  is a norm defined by Lifshits and Linde \cite{LifshitsLinde2005}. When $\frac{1}{r}=\tau+\frac{1}{q}$, $\|w\|_{L^r(0,\infty)}\le \|w\|_{r,\tau,q}$. 
 
 In this subsection, we present a similar general result for $J_{m,\bm\alpha}(I_{\gamma}B_H)(t)$ by applying  the well-known conjecture of the Gaussian correlation inequality. This conjecture states that for any centered Gaussian random element $ X $ and any symmetric, convex sets $ A $ and $ B $ in a separable Banach space $ E$, the following inequality holds:
\begin{align}\label{eq:GaussCor} \pr( X\in A\cap B)\ge \pr (X\in   A) \pr (X\in   B). 
\end{align}
The Gaussian correlation inequality conjecture has been proven by Royen \cite{Royen2014} (c.f. Lata\l a and Matak \cite{LatalaMatlak}). The relation \eqref{eq:GaussCorWeak} serves as a weaker form of \eqref{eq:GaussCor}.
  
\begin{theorem} \label{th:3.1}Let $m\ge 0$, $\gamma\ge 0$, $\alpha_1+\cdots+\alpha_i<H+i+\gamma$, $i=1,\ldots,m$,  $\tau=H+m+\gamma$. Suppose  $q\ge 1$, $\frac{1}{r}=H+m+\gamma+\frac{1}{q}$, and  $\|w\|_{r,\tau,q}<\infty$. Then
\begin{align}\label{eq:th:3.1:1}
 &\lim_{\epsilon\to 0}  \epsilon^{1/\tau}\log \pr\left(\left\|w(t)t^{\alpha_1+\cdots+\alpha_m}  J_{m,\bm\alpha}(I_{\gamma}B_H)(t) \right\|_{L^q(0,\infty)}<a_H \epsilon\right)\nonumber\\
 =&\lim_{\epsilon\to 0}  \epsilon^{1/\tau}\log \pr\left(\left\|w(t)t^{\alpha_1+\cdots+\alpha_m}  J_{m,\bm\alpha}(W_{H+\gamma})(t) \right\|_{L^q(0,\infty)}< \epsilon\right)\nonumber\\
=& \lim_{\epsilon\to 0} \epsilon^{1/\tau}\log \pr\left(\left\|wW_{\tau}\right\|_{L^q(0,\infty)}< \epsilon\right) 
= -\kappa(\tau,q)\|w\|_{L^r(0,\infty)}^{1/\tau}, 
\end{align}
where $\kappa(\lambda,q)$ is defined as \eqref{eq:constant:2}.  When $q=\infty$, we assume that $w$  is almost everywhere continuous.  
\end{theorem} 

To prove Theorem \ref{th:3.1}, we need the following lemma on the small ball probabilities of a self-similar Gaussian process under the weighted $L^q$-norm, which is based on the Gaussian correlation inequality \eqref{eq:GaussCor}.
\begin{lemma}\label{lem:3.6}  Let $  X(t)$ be a continuous, centered, self-similar Gaussian process of index $\tau>0$.
Suppose  that
\begin{equation}\label{eq:lem:3.6:1}-\log \pr\Big(\sup_{0\le t\le 1} \left|  X(t)\right|<\epsilon\Big)\le c_0\epsilon^{-\beta}+o(\epsilon^{-\beta}) \;\text{ as } \epsilon\to 0,  
\end{equation} 
where $0<\beta<\infty$.  Then
\begin{equation}\label{eq:lem:3.6:2}-\log \pr\left( \left \|w X\right \|_{L^q(0,\infty)}<\epsilon\right)\le 
\begin{cases} C\|w\|_{r,\tau,q}^{\beta}\cdot \epsilon^{-\beta}, & \text{for all }\epsilon>0, \\
 c_0\|w\|_{r,\tau,q}^{\beta}\cdot \epsilon^{-\beta}+o(\epsilon^{-\beta})  &\text{as } \epsilon\to 0,
 \end{cases}
\end{equation}
where $\frac{1}{r}=\frac{1}{\beta}+\frac{1}{q}$ and the constant $C$ does not depend on $w$. 
\end{lemma}

{\bf Proof.} Without loss of generality, we suppose  $\|w\|_{r,\tau,q}=1$.  Write  $\Delta_k=(2^{k-1},2^k]$. By \eqref{eq:lem:3.6:1}, for any $C_1>c_0$, there exists an $\epsilon_0\in (0,1)$ such that 
$$-\log \pr\Big(\|X\|_{[0,1]}<\epsilon\Big)\le C_1\epsilon^{-\beta}, \;\; 0<\epsilon\le \epsilon_0. $$
On the other hand, by the isoperimetric concentration inequality (c.f. Lemma 3.1  of Ledoux and Talagrand \cite{LT1991}) and the elememary inequality that $-\log (1-x)\le 2 x$ ($1/2\le x\le 1$), we have that  
\begin{align*}
-\log \pr & \Big(\|X\|_{[0,1]}<\epsilon\Big)\le   2\pr\Big(\|X\|_{[0,1]}\ge \epsilon\Big)\le 2\pr\Big(\|X\|_{[0,1]}-\text{med}(\|X\|_{[0,1]})\ge \epsilon/2\Big)\\
&\le   2\exp\left\{-\frac{\epsilon^2}{8\sup_{0\le t\le 1}\ep X^2(t)}\right\}= 2\exp\left\{-\frac{\epsilon^2}{8\ep X^2(1)}\right\},\;\; \epsilon\ge 2 \text{med}(\|X\|_{[0,1]}), 
\end{align*}
where $\text{med}(\|X\|_{[0,1]})$ is the median of $\|X\|_{[0,1]}$. Thus, for any $\beta_0\le \beta$ there exists an constant $C_{\beta_0,\epsilon_0}>0$ such that
\begin{equation}\label{eq:proof:lem:3.6:1}
-\log \pr\Big(\|X\|_{[0,1]}<\epsilon\Big)\le C_{\beta_0,\epsilon_0}\epsilon^{-\beta_0}, \;\;  \epsilon> \epsilon_0. 
\end{equation}
Notice that $\|w\|_{L^q(\Delta_k)}\|X\|_{\Delta_k}\overset{d}=\|w\|_{L^q(\Delta_k)}2^{k\tau}\|X\|_{[1/2,1]}$
and 
$$\|wX\|_{L^q(0,\infty)}\le \begin{cases} \Big(\sum_{k=-\infty}^{\infty} \|w\|_{L^q(\Delta_k)}^q\|X\|_{\Delta_k}^q\Big)^{1/q}, & 1\le q<\infty, \\
\max_k \|w\|_{L^q(\Delta_k)}\|X\|_{\Delta_k}, & q=\infty. \end{cases}
$$
Without loss of generality, we assume that $\|w\|_{L^q(\Delta_k)}\ne 0$ for all $k$.  Write $a_k=\|w\|_{L^q(\Delta_k)}2^{k\tau}$, $\lambda_k=a_k^r$. Then $\sum_k\lambda_k=\|w\|_{r,\tau,q}^r=1$. By repeating the Gaussian correlation inequality \eqref{eq:GaussCor}, we have that
\begin{align*}
 -\log \pr & \Big(\|wX\|_{L^q(0,\infty)}<\epsilon\Big)\le  -\log \pr\Big(\bigcap_{k=-\infty}^{\infty} \Big\{\|w\|_{L^q(\Delta_k)}\|X\|_{\Delta_k}<\lambda_k^{1/q}\epsilon\Big\}\Big)\\
\le & -\sum_{k=-\infty}^{\infty}\log \pr\Big( \|w\|_{L^q(\Delta_k)}\|X\|_{\Delta_k}<\lambda_k^{1/q}\epsilon\Big) 
\le  -\sum_{k=-\infty}^{\infty}\log \pr\Big(\|X\|_{[0,1]}<\lambda_k^{1/q}a_k^{-1}\epsilon\Big)\\
\le &    C_1 \sum_{k:\;\epsilon\lambda_k^{1/q}a_k^{-1}\le \epsilon_0 } \big(\epsilon \lambda_k^{1/q}a_k^{-1}\big)^{-\beta}
+C_{\beta,\epsilon_0}  \sum_{k:\;\epsilon\lambda_k^{1/q}a_k^{-1}> \epsilon_0 } \big(\epsilon \lambda_k^{1/q}a_k^{-1}\big)^{-\beta}\\
\le &    C_1\epsilon^{-\beta}  \sum_{k=-\infty}^{\infty}a_k^r
+C_{\beta,\epsilon_0}\epsilon^{-\beta} \sum_{k:a_k^{r/\beta}< \epsilon/\epsilon_0 } a_k^r 
= \begin{cases} (C_1+C_{\beta,\epsilon_0})\epsilon^{-\beta}, & 
\epsilon>0, \\
 C_1\epsilon^{-\beta} +o\big( \epsilon^{-\beta}\big), & \epsilon\to 0,
\end{cases}
\end{align*}
since $\sum_{k=-\infty}^{\infty} a_k^r=1$ and $\sum_{k:a_k^{r/\beta}< \epsilon/\epsilon_0 } a_k^r\to 0$ as $\epsilon\to 0$.  $\Box$

\bigskip

{\bf Proof of Theorem \ref{th:3.1}}. Let
\begin{align*}
X(t)=& a_H^{-1}t^{\alpha_1+\cdots+\alpha_m}    J_{m,\bm\alpha}(I_{\gamma}B_H)(t)-W_{H+m+\gamma}(t) \\
& \text{ or } t^{\alpha_1+\cdots+\alpha_m}    J_{m,\bm\alpha}(W_{H+\gamma})(t)-W_{H+m+\gamma}(t) . 
\end{align*}
By \eqref{eq:smallbofZ} and \eqref{eq:proof:th:2:1}, 
\begin{equation}\label{eq:proof:th:3.1:0}-\log \pr\left(\|X\|_{[0,1]}< \epsilon\right)\preccurlyeq \epsilon^{-1/(H+m+1+\gamma)}= o\big(\epsilon^{-1/(H+m+\gamma)}\big).
\end{equation}
Notice that $X(t)$ is a continuous, centered, self-similar Gaussian process of index $\tau=H+m+\gamma$. Let $\beta=1/\tau$. Since $\|w\|_{r,\tau,q}<\infty$,  by applying Lemma \ref{lem:3.6} we obtain
$$-\log \pr\left(\|wX\|_{L^q(0,\infty)}< \epsilon\right)= o\big(\epsilon^{-1/(H+m+\gamma)}\big).
$$  
Thus, the first and the second equalities of \eqref{eq:th:3.1:1} holds. For the last equality, it is sufficient to show that
\begin{equation}\label{eq:proof:th:3.1:1}
 \lim_{\epsilon\to 0}\epsilon^{1/\lambda}\log\pr\Big(\|wW_{\lambda}\|_{L^q(0,\infty)}<\epsilon\Big)
=-\kappa(\lambda,q)\|w\|_{L^r(0,\infty)}^{1/\lambda}, 
\end{equation}
if $\|w\|_{r,\lambda,q}<\infty$ with $\frac{1}{r}=\lambda+\frac{1}{q}$, where, the function $w$   is assumed almost everywhere continuous when  $q=\infty$.

The proof of \eqref{eq:proof:th:3.1:1} is similar to that of Theorem 4.6 of Lifshits and Linde \cite{LifshitsLinde2005} where \eqref{eq:proof:th:3.1:1} is proved for $0<\lambda<1$. We   consider the case of $q<\infty$. Since $\|w\|_{r,\lambda,q}<\infty$, for any given $\delta > 0$ we may split $w$ into the sum of two   functions $w^{(1)}$ and $w-w^{(1)}$ such that $\|w-w^{(1)}\|_{r,\lambda,q}<\delta$, and $w^{(1)}\in L^q(\Delta)$ for a bounded and closed interval $\Delta\subset (0,\infty)$. For $w^{(1)}$, since $w^{(1)}\in L^q(\Delta)$, we also may split it into the sum of $w^{(2)}$ and $w^{(1)}-w^{(2)}$ such that $\|w^{(1)}-w^{(2)}\|_{r,\lambda,q}<\delta$, and that $w^{(2)}$ is an interval step function of the form
$$ w^{(2)}=\sum_{j=1}^m w_j \mathbbm{I}_{(s_j,s_{j+1}]}, \; s_1<s_2<\cdots<s_{m+1}$$
(c.f. Lemma 4.4 of Lifshits and Linde \cite{LifshitsLinde2002}). 
Here and in the sequel, $\mathbb{I}_{\Delta}$ is the indicator function of the set $\Delta$. 
\red{In fact, suppose $\Delta\subset (2^{-k_0}, 2^{k_0}]$. Since continuous functions are dense in $L^q(\Delta)$ and a continuous function on a closed interval can be uniformly approximated by an  interval step function, we can find an interval step function $w^{(2)}$ such that
$$ \|w^{(1)}-w^{(2)}\|_{L^q(\Delta)}^r<\delta^r/2^{k_0(r\lambda+2)}. $$
Then
$$\|w^{(1)}-w^{(2)}\|_{r,\lambda,q}^r\le \sum_{k=-k_0+1}^{k_0} 2^{kr\lambda}\delta^r/2^{k_0(r\lambda+2)} \le \delta^r. $$ 
} Thus,
$$ \|w-w^{(2)}\|_{L^r(0,\infty)}\le \|w-w^{(2)}\|_{r,\lambda,q}<2\delta. $$
By Lemma \ref{lem:3.6}, 
\begin{align*} & \liminf_{\epsilon\to 0}\epsilon^{1/\lambda}\log\pr\Big(\|(w-w^{(2)})W_{\lambda}\|_{L^q(0,\infty)}<\epsilon\Big) \\
\ge & -\kappa(\lambda,q)\|w-w^{(2)}\|_{r,\lambda,q}^{1/\lambda}>-\kappa(\lambda,q)(2\delta)^{1/\lambda}. 
\end{align*}
Notice by \eqref{eq:GaussCor}, for any $0<\lambda_0<1$,
\begin{align*}& \pr\Big(\|wW_{\lambda}\|_{L^q(0,\infty)}<\epsilon\Big)\ge   \pr\Big(\|(w-w^{(2)})W_{\lambda}\|_{L^q(0,\infty)}<\lambda_0\epsilon, \|w^{(2)}W_{\lambda}\|_{L^q(0,\infty)}<(1-\lambda_0) \epsilon\Big) \\
\ge &  \pr\Big(\|(w-w^{(2)})W_{\lambda}\|_{L^q(0,\infty)}<\lambda_0\epsilon\Big) \pr\Big(\|w^{(2)}W_{\lambda}\|_{L^q(0,\infty)}<(1-\lambda_0) \epsilon\Big) 
\end{align*}
and
\begin{align*}& \pr\Big(\|w^{(2)}W_{\lambda}\|_{L^q(0,\infty)}<  (1+\lambda_0)\epsilon\Big)
\ge   \pr\Big(\|(w-w^{(2)})W_{\lambda}\|_{L^q(0,\infty)}<\lambda_0\epsilon\Big) \pr\Big(\|wW_{\lambda}\|_{L^q(0,\infty)}<  \epsilon\Big). 
\end{align*}
Thus,  without loss of generality we can assume that $w=\sum_{j=1}^m w_j \mathbbm{I}_{(s_j,s_{j+1}]}$ is an interval step function. Then
$$wW_{\lambda}=w\widetilde{W}_{\lambda}+\sum_{j=1}^m w_jR_{j}^{\lambda}, $$ 
where 
$$ \widetilde{W}_{\lambda}(t)=\frac{1}{\Gamma(\lambda+1/2)}\int_0^{s_j}(t-s)^{\lambda-1/2}d B(s), \; t\in (s_j,s_{j+1}], $$
$$ R_{j}^{\lambda}(t)=\frac{1}{\Gamma(\lambda+1/2)}\int_{s_j}^{t}(t-s)^{\lambda-1/2}d B(s), \; t\in (s_j,s_{j+1}]. $$
Notice that $\|R_j^{\lambda}\|_{L^q(s_j,s_{j+1}]}\overset{d}=(s_{j+1}-s_j)^{\lambda+1/q}\|W_{\lambda}\|_{L^q(0,1]}$.
 We have
$$ \lim_{\epsilon\to 0} \epsilon^{1/\lambda} \log \pr\left(|w_j|\|R_j^{\lambda}\|_{L^q(s_j,s_{j+1}]}<\epsilon\right)
=-\kappa(\lambda,q) |w_j|^{1/\lambda}(s_{j+1}-s_j)^{1+1/(q\lambda)}. $$
By noticing that $\|R_j^{\lambda}\|_{L^q(s_j,s_{j+1}]}$, $j=1,\cdots,m$, are independent, with the same argument of (4.14) of Lifshits and Linde \cite{LifshitsLinde2005}, we have that
\begin{align*}
&\lim_{\epsilon\to 0}\epsilon^{1/\lambda}\log\pr\Big(\|\sum_{j=1}^m w_jR_{j}^{\lambda}\|_{L^q(0,\infty)}<\epsilon\Big)\\
=&\lim_{\epsilon\to 0}\epsilon^{1/\lambda}\log\pr\Big(\Big(\sum_{j=1}^m |w_j|^q\|R_{j}^{\lambda}\|_{L^q(s_j,s_{j+1}]}^q\Big)^{1/q}<\epsilon\Big)\\
=&-\kappa(\lambda,q)\Big(\sum_j |w_j|^{\frac{1/\lambda}{1+1/(q\lambda)}}(s_{j+1}-s_j)\Big)^{1+1/(q\lambda)}=-\kappa(\lambda,q)\|w\|_{L^r(0,\infty)}^{1/\lambda}.
\end{align*}
 
On the other hand, by Lemmas \ref{lem:6} and \ref{lem:ChungLIL1}, 
\begin{align*}
&-\log\pr\Big(\|w\widetilde{W}_{\lambda}\|_{L^q(0,\infty)}<\epsilon\Big)\\
\le & -\log\pr\Big( \max_{j\le m} |w_j|\sup_{s_j<t\le s_{j+1}} \Big|\frac{1}{\Gamma(\lambda+1/2)}\int_0^{s_j}(t-s)^{\lambda-1/2}d B(s)\Big|<\epsilon\Big) 
  =o(\epsilon^{-1/\lambda}).
\end{align*}
By Lemme \ref{lem:ChungLIL1} again, \eqref{eq:proof:th:3.1:1} holds. 

 \red{
 When $q=\infty$,   $ r=1/\lambda$.   Of course, it suffices to verify \eqref{eq:proof:th:3.1:1} for a weight function $w$ with a support on   a bounded and closed interval $\Delta\subset (0,\infty)$. As assumed, $w$ is  almost everywhere continuous and so $|w|^r$ is  Riemann integrable on $\Delta$. It follows that there are two   nonnegative  interval step functions $w^{(i)}=\sum_{j=1}^m w_j^{(i)} \mathbbm{I}_{(s_j,s_{j+1}]}$,  $i=1,2$, such that 
 $$ w^{(1)} \le |w|\le w^{(2)} $$
 and 
$$\|w\|_{L^r(\Delta)}^{r} -\delta\le  \|w^{(1)}\|_{L^r(\Delta)}^{r},\;\;  \|w^{(2)}\|_{L^r(\Delta)}^{r} \le \|w\|_{L^r(\Delta)}^{r} +\delta. $$ 
For the interval step functions, we have
\begin{align*}
&\lim_{\epsilon\to 0}\epsilon^{1/\lambda}\log\pr\Big(\|w^{(i)} W_{\lambda}\|_{L^{\infty}(\Delta)}<\epsilon\Big)=\lim_{\epsilon\to 0}\epsilon^{1/\lambda}\log\pr\Big(\|w^{(i)}  W_{\lambda}\|_{\Delta}<\epsilon\Big)\\
=&\lim_{\epsilon\to 0}\epsilon^{1/\lambda}\log\pr\Big(\|\sum_{j=1}^m w_j^{(i)} R_{j}^{\lambda}\|_{\Delta}<\epsilon\Big) 
= \lim_{\epsilon\to 0}\epsilon^{1/\lambda}\log\pr\Big(\max_{1\le j\le m}|w_j^{(i)} |\|R_{j}^{\lambda}\|_{(s_j,s_{j+1}]}<\epsilon\Big) \\
= &\lim_{\epsilon\to 0}\epsilon^{1/\lambda}\sum_{j=1}^m\log\pr\Big( |w_j^{(i)}|\|R_{j}^{\lambda}\|_{(s_j,s_{j+1}]}<\epsilon\Big) \\
=&-\kappa_{\lambda}\sum_{j=1}^m |w_j^{(i)}|^{1/\lambda}(s_{j+1}-s_j) 
= -\kappa_{\lambda}\|w^{(i)}\|_{L^r(\Delta)}^{r},  
\end{align*}
by the independence.  Thus \eqref{eq:proof:th:3.1:1} holds. }
The proof is now completed.  $\Box$.

\subsection{Small ball probabilities for the weighted integrals with a general weight}

  Duker,  Li and  Linde \cite{DLL2000}  showed 
  $$ -\log \pr\left(\sup_{0<t\le 1} \left|\int_0^tw(s)B(s)\right|<\epsilon\right) \approx \epsilon^{-2/3}, $$
  under suitable conditions placed  on the weight function $w$. 
In this subsection, we consider the precise small ball probability of  the weighted integrals of $J_{m,\bm\alpha}(I_{\gamma}B_H)(t)$ with a general weight function $w(t)$. 
\begin{theorem} \label{th:3.2}   Let $\Delta=(a,b)\subset (0,\infty)$,  $m\ge 0$, $\gamma\ge 0$, $\alpha_1+\cdots+\alpha_i<H+i+\gamma$, $i=1,\ldots,m$, and, $\tau=H+m+\gamma$.  Suppose   $q>1$,  
$\frac{1}{r}=\tau+1+\frac{1}{q}$,  and 
\begin{equation}\label{eq:th:3.2:cond}\int_0^t |w(s)|s^{\tau}ds<\infty\; \text {and }\; \|w\mathbbm{I}_{\Delta}\|_{r,\tau+1,q} =\left(\sum_{k=-\infty}^{\infty}2^{kr(\tau+1)} \|w\mathbbm{I}_{\Delta}\|^r_{L^q(2^{k-1},2^k]} \right)^{1/r}  <\infty.
\end{equation} 
Moreover, we assume that      $\Delta$ is bounded  when $1<q<\infty$. 
  Then
\begin{align}\label{eq:th:3.2:1}
 & \lim_{\epsilon\to 0}   \epsilon^{1/(\tau+1)}\log \pr\left(\left\|\int_0^tw(s)s^{\alpha_1+\cdots+\alpha_m}  J_{m,\bm\alpha}(I_{\gamma}B_H)(s) \right\|_{L^q(\Delta)}<a_H \epsilon\right)\nonumber\\
=& \lim_{\epsilon\to 0} \epsilon^{1/(\tau+1)}\log \pr\left(\left\|\int_0^tw(s)s^{\alpha_1+\cdots+\alpha_m} J_{m,\bm\alpha}(W_{H+\gamma})(s)ds\right\|_{L^q(\Delta)}< \epsilon\right) \nonumber\\
=&\lim_{\epsilon\to 0} \epsilon^{1/(\tau+1)}\log \pr\left(\left\|wW_{\tau+1}\right\|_{L^q(\Delta)}< \epsilon\right) 
=-\kappa(\tau+1,q)\|w\|_{L^r(\Delta)}^{1/(\tau+1)}, 
\end{align}
where $\kappa(\lambda,q)$ is defined as \eqref{eq:constant:2}.   In particular, 
$$ \lim_{\epsilon\to 0} \epsilon^{2/3}\log \pr\left(\sup_{t>0}\left|\int_0^tw(s)B(s) ds\right|< \epsilon\right) 
= -\kappa_{\frac{3}{2}} \int_0^{\infty}|w(s)|^{2/3}ds $$ 
whenever  $\sum_{k=-\infty}^{\infty}2^{k} \underset{t\in (2^{k-1},2^k]}{\text{ess sup}}|w(t)|^{2/3} <\infty$,  and
$$ \lim_{\epsilon\to 0} \epsilon^{2/3}\log \pr\left(\left\|\int_0^tw(s)B(s) ds\right\|_{L^2(0,1]}< \epsilon\right) 
= -\frac{3}{8} \left(\int_0^1|w(s)|^{1/2}ds\right)^{4/3} $$ 
whenever $\sum_{k=-\infty}^{0}2^{3k/4} \big(\int_{2^{k-1}}^{2^k}|w(t)|^2dt\big)^{1/4} <\infty$ and $\int_0^1|w(s)|s^{1/2}ds<\infty$. 
\end{theorem} 

\begin{remark} If $|w|^r$ is Riemann-integrable on $(0,\infty)$, and there exists a  constant $C\ge 0$ such that 
\begin{align}\label{eq:remark:3.1:1} |w(x)|\le C\big[|w(y/2)|+|w(2z)|\big] \text{ for all } x,y,z\in (2^{k-1},2^k], 
\end{align} 
when $k>0$ is large enough, and when $-k>0$ is large enough if $w$ is unbounded in a neighborhood of zero, then $\|w\|_{r,\tau+1,q}<\infty$.  In fact, it is sufficient to notice that
\begin{align*}
&2^{kr(\tau+1)}\|w\|_{L^q(\Delta_k)}^r\le 2^{kr(\tau+1)}\sup_{t\in \Delta_k} |w(t)|^r|\Delta_k|^{r/q}\\
\le & C2^{k}\big(\inf_{y\in \Delta_k}|w(y/2)|^r+\inf_{z\in \Delta_k}|w(2z)|^r\big)\le C\big(\|w\|_{L^r(\Delta_{k-1})}^r +\|w\|_{L^r(\Delta_{k+1})}^r\big), \end{align*}
where $\Delta_k=(2^{k-1},2^k]$. 

When $|w(t)|$ is convex or $|w(t)|t^{\alpha}$ is monotonic, \eqref{eq:remark:3.1:1} is satisfied.

In fact, if $|w(t)|$ is convex, then 
\begin{align*}
& |w(x)|= \left|w\big(\frac{y}{2} \lambda+2z (1-\lambda)\big)\right| \\
\le & \lambda \left|w\big(\frac{y}{2}\big)\right|+(1-\lambda)\left|w\big(2z\big)\right|
\le  \left|w\big(\frac{y}{2}\big)\right|+\left|w\big(2z\big)\right|,
\end{align*}
where $\lambda=\frac{4z-2x}{4z-y}\in [0,1]$. If $|w(t)|t^{\alpha}$ is non-decreasing, then
$|w(x)|/x^{\alpha}\le |w(2z)|/(2z)^{\alpha}$, and so
$$ |w(x)|\le \frac{x^{\alpha}}{(2z)^{\alpha}}|w(2z)|\le 4^{|\alpha|}|w(2z)|. $$
If $|w(t)|t^{\alpha}$ is non-increasing, then
$|w(x)|/x^{\alpha}\le |w(y/2)|/(y/2)^{\alpha}$, and so
$$ |w(x)|\le \frac{x^{\alpha}}{(y/2)^{\alpha}}|w(y/2)|\le 4^{|\alpha|}|w(y/2)|. $$
\end{remark}  
\begin{remark} When $q<\infty$, we need to assume that $\Delta=(a,b)$ is bounded because $\|\int_0^t Y(s)ds\|_{L^q(0,\infty)}$ is  not finite. 
\end{remark}  
\begin{remark}
Since $I(wX)(t)$ is continuous, $\|I(wX)\|_{L^{\infty}(\Delta)}=\|I(wX)\|_{\Delta}$. 
\end{remark} \red{
\begin{remark} When $q=\infty$, $\int_0^t|w(s)|s^{\tau}ds<\infty$ is implied by $\|w\|_{r,\tau+1,q}<\infty$. In fact, in this case $1/r=\tau+1>1$ and
\begin{align*}
 \left(\int_0^t|w(s)|s^{\tau}ds\right)^r
 \le & \left(\sum_{k=-\infty}^{\infty} \int_{2^{k-1}}^{2^k} |w(s)|s^{\tau}ds\right)^r
\le \left(\sum_{k=-\infty}^{\infty}2^{k(\tau+1)}\|w\|_{L^{\infty}(2^{k-1},2^k]}\right)^r\\
 \le &\sum_{k=-\infty}^{\infty}2^{k(\tau+1)r}\|w\|_{L^{\infty}(2^{k-1},2^k]}^r 
 =\left(\|w\|_{r,\tau+1,q}\right)^r.
 \end{align*}
\end{remark}
}

Theorem \ref{th:3.2} will follow from Theorem \ref{th:3.1} and the following proposition. 

\begin{proposition}\label{lem:3.7}  Let $  X(t)$ be a continuous, centered, self-similar Gaussian process of index $\tau>0$.
Suppose that $1< q\le\infty$ and
\begin{equation}\label{eq:lem:3.7:1}-\log \pr\Big(\sup_{0\le t\le 1} \left|  X(t)\right|<\epsilon\Big)\le c_0\epsilon^{-1/\tau}+o(\epsilon^{-1/\tau}).   
\end{equation} 
 Assume that  $\int_0^t|w(s)|s^{\tau}ds<\infty$, and $\Delta$ is a sub-interval of $(0,\infty)$. Suppose that $\Delta$ is bounded   when $q<\infty$.
 Then   
\begin{equation}\label{eq:lem:3.7:2}-\log \pr\left( \left\|I\big(wX\big)\right\|_{L^q(\Delta)}<\epsilon\right)\le C\|w\mathbbm{I}_{\Delta}\|_{r,\tau+1,q}^{1/(\tau+1)}\cdot \epsilon^{-1/(\tau+1)}+o\big(\epsilon^{-1/(\tau+1)}\big),
\end{equation}
where $\frac{1}{r}=\tau+1+\frac{1}{q}$ and the constant $C>0$ does not depend on $w$.   Moreover,  we assume that $w$ is almost everywhere continuous on $\Delta$ when $q=\infty$, and $\Delta$ is bounded when $1<q<\infty$. Then
\begin{equation}\label{eq:lem:3.7:3}\log \pr\left( \left\|I\big(wX\big)-wI(X)\right\|_{L^q(\Delta)}<\epsilon\right)=o\big(\epsilon^{-1/(\tau+1)}\big) \text{ as }  \epsilon\to 0,
\end{equation}
if $\|w\mathbbm{I}_{\Delta}\|_{r,\tau+1,q}<\infty$. 
\end{proposition}

To prove Proposition \ref{lem:3.7}, we need a some lemma.

\begin{lemma}\label{lem:3.8}  Let X be a centered  Gaussian random element in a separable Banach space $(E,\|\cdot\|_E)$ and suppose
that
$$ -\log \pr(\|X\|<\epsilon) \le c_0\epsilon^{-\beta}, 0<\epsilon\le \epsilon_0 $$
where $\beta>0$. For an operator $T$ from $E$  into another
Banach space $(F,\|\cdot\|_F)$, we denote the $n$th dyadic entropy number of $T$ as
$$ e_n(T)=\inf\Big\{\epsilon>0: T(U_E)\subset\bigcup_{j=1}^{2^{n-1}}\{x_j+\epsilon U_F\}, x_j\in U_F \Big\}, $$
where $U_E$ and $U_F$ are the unit balls in $E$ and $F$, respectively. If 
$$ e_n(T)\le C_1n^{-\gamma}, \;  n\ge n_0,
$$ 
for some $\gamma>0$, then there exist constants $C_2>0$ and $\epsilon_2>0$ which depend only on $c_0$, $\epsilon_0$, $\beta$, $\gamma$, $n_0$ and $C_1$, such that
$$  -\log \pr(\|T(X)\|_F<\epsilon) \le   C_2 \epsilon^{-\beta/(\beta\gamma+1)}, \;\; 0<\epsilon\le \epsilon_2. $$
  \end{lemma}
This lemma follows from Theorem 5.2 of Li and Linde \cite{LiLinde1999}. It is sufficient to notice that, in 
Li and Linde's proof,  the choice of positive constants $C_2$ and $\epsilon_2$ does not depend  on $T$ and $X$ (See the proofs of Theorem 1 of  Kuelbes and Li \cite{KueblesLi1993}  and Theorem 1.2 of  Li and Linde \cite{LiLinde1999}, c.f. also Creutzig\cite{Creutzig1999}).

{\bf Proof of Proposition \ref{lem:3.7}.} Under the condition $\int_0^{t}|w(s)|s^{\tau}ds<\infty$, $I(wX)(t)$ is finite and a continuous, centered, Gaussian process.  
Let $\Delta=(a,b)$. Then 
$$\left\|I\big(wX-w\mathbbm{I}_{\Delta}X\big)\right\|_{L^q(\Delta)}=\Big|\int_0^aw(s)X(s)ds\Big|\cdot|\Delta|^{1/q} $$
and
$$ \pr\left(\Big|\int_0^aw(s)X(s)ds\Big|\cdot|\Delta|^{1/q}<\epsilon\right)\approx \epsilon, $$
where $|\Delta|^{1/q}=1$ when $q=\infty$. Thus, by Lemma \ref{lem:ChungLIL1}, without loss of generality   we can assume that $w$ has a support on $\Delta$ and $w=w\mathbbm{I}_{\Delta}$.

 (i)  For \eqref{eq:lem:3.7:2}, 
without loss of generality, we assume  $\|w\|_{r,\tau+1,q}=1$.   
 We first consider the case of  $q=\infty$ and use the arguments of   Duker,  Li and  Linde \cite{DLL2000}. Recall $\frac{1}{r}=\tau+1+\frac{1}{q}=\tau+1$. Let $\frac{1}{r_0}=\tau+\frac{1}{2}$, $\psi_1(t)=|w(t)|^{r/r_0}\text{sgn}(w(t))$ and $\psi_2=|w|^{1/(2\tau+2)}$. Then 
 $w=\psi_1\psi_2$,
 $$ \int_0^{\infty}\psi_2^2(s)ds =\int_0^{\infty}|w(s)|^rds \le \|w\|_{r,\tau+1,q}^r= 1 $$
 and
 $$ \|\psi_1\|_{r_0,\tau,2}^{r_0}\le \sum_k 2^{kr_0\tau}\big(\|\psi_1\|_{\Delta_k} |\Delta_k|^{1/2} \big)^{r_0} \le \sum_k 2^k \|w\|_{\Delta_k}^{r}=\|w\|_{r,\tau+1,q}^r=1. $$
 By Lemma  \ref{lem:3.6}, 
 $$ -\log\pr\left(\|\psi_1X\|_{L^2(0,\infty)}<\epsilon\right)\le  C_1\epsilon^{-1/\tau}, \; \epsilon>0, $$
 where the constant $C_1$ does not depend on $\psi_1$.
 Consider an operator  $I_{\psi_2}$ as
 $$I_{\psi_2}:L^2(0,\infty)\to C(0,\infty) \text{ with } I_{\psi_2}f(t)= \int_0^t\psi_2(s)f(s)ds.$$
 Notice that $\int_0^{\infty}\psi_2^2(s)ds\le 1$. We have  $e_n(I_{\psi_2})\le c_1n^{-1}$ by Lemma 6 and its proof of Duker,  Li and  Linde \cite{DLL2000} (c.f. Theorem 4.6 (1) of Lifshits and Linde \cite{LifshitsLinde2002}).
 Notice that $ I(wX)= I_{\psi_2} (\psi_1X)$. By Lemma \ref{lem:3.8}, there exist positive constants $C_3=C_3(C_1,c_1,\beta)$ and $\epsilon_0=\epsilon_0(C_1,c_1,\beta)$, such that
 $$ -\log\pr\left(\|I(wX)\|_{(0,\infty)}<\epsilon\right)\le C_3\epsilon^{-1/(\tau+1)},\;\; 0<\epsilon\le \epsilon_0. $$
 Therefore, \eqref{eq:lem:3.7:2} holds for $q=\infty$.

 Next, we suppose  that $1< q<\infty$, and 
 $\Delta\subset (0,2^{k_0}]$.  Write  $\Delta_k=(2^{k-1},2^k]$ and $a_k=\|w\|_{L^q(\Delta_k)}2^{k(\tau+1)}$.   Define  $X_k(t)=\frac{X(2^kt)}{2^{k\tau}}$,
 $ \xi_k=2^{k/q}\int_0^{2^{k-1}}w(s)X(s)ds$, and
$$ w_k(t)=\begin{cases} \frac{w(2^kt)2^{k/q}}{\|w\|_{L^q(\Delta_k)}}, & t\in(1/2,1]\\
0, & t\in [0,1/2]
\end{cases}  $$ 
 when $\|w\|_{L^q(\Delta_k)}\ne 0$,  and $ w_k(t)=0$   when $\|w\|_{L^q(\Delta_k)}= 0$. 
   Then $X_k\overset{d}=X$, $\|w_k\|_{L^q[0,1]}\le 1$,  and
\begin{align}\label{eq:proof:lem:3.7:1}
&\|I\big(wX\big)\|_{L^q(\Delta)}= \Big(\sum_{k\le k_0}   \|I\big(wX\big)\|_{L^q(\Delta_k)}^q\Big)^{1/q}\nonumber\\
\le & \Big(\sum_{k=-\infty}^{\infty}   \big\|\int_{2^{k-1}}^{t} w(s)X(s)ds\big\|_{L^q(\Delta_k)}^q\Big)^{1/q}+\Big(\sum_{k\le k_0}|\xi_k|^q\Big)^{1/q}\nonumber\\
= & \Big(\sum_{k=-\infty}^{\infty} a_k^q  \big\|I \big(w_kX_k\big) \big\|_{L^q[0,1]}^q\Big)^{1/q}+\Big(\sum_{k\le k_0}|\xi_k|^q\Big)^{1/q}:=\eta_1+\eta_2.
\end{align}
By repeating the Gaussian correlation inequality \eqref{eq:GaussCor}, we have that for $\lambda_k=a_k^r$,
\begin{align}\label{eq:proof:lem:3.7:2}
 -\log \pr & \Big(\eta_1<\epsilon\Big)\le  -\log \pr\Big(\bigcap_{k=-\infty}^{\infty} \Big\{ a_k  \big\|I\big(w_kX_k\big) \big\|_{L^q[0,1]}<\lambda_k^{1/q}\epsilon\Big\}\Big)\nonumber\\
\le & -\sum_{k=-\infty}^{\infty}\log \pr\Big(   \big\|I\big(w_kX_k\big) \big\|_{L^q[0,1]}<\lambda_k^{1/q}a_k^{-1}\epsilon\Big)\nonumber\\
=& -\sum_{k=-\infty}^{\infty}\log \pr\Big(   \big\|I\big(w_kX\big) \big\|_{L^q[0,1]}<a_k^{-r(\tau+1)}\epsilon\Big)\nonumber\\
\le &C\epsilon^{-1/(\tau+1)}\sum_{k=-\infty}^{\infty}a_k^r=C\epsilon^{-1/(\tau+1)},\;\; \epsilon>0. 
\end{align}
where the last inequality is due to the following fact:
\begin{equation}\label{eq:proof:lem:3.7:4} -\log \pr\Big(   \big\|I\big(w_kX\big) \big\|_{L^q[0,1]}< \epsilon\Big)\le C\epsilon^{-1/(\tau+1)}, 
\end{equation}
for all $\epsilon>0$ and all $k$.  

For verifying \eqref{eq:proof:lem:3.7:4}, we let 
$$ T_kf(t)=I(w_kf)(t)=\int_0^t w_k(s) f(s) ds. $$
We write $T_k:C[0,1]\to L^q[0,1]$ as $T=  I \circ    S_{w_k}$, where $S_{w_k}:C[0,1]\to L^q[0,1]$ and  $I:L^q[0,1]\to L^q[0,1]$  are defined as
$$ S_{w_k}f(t)=w_k(t)f(t),  \;\; If(t)=\int_0^t f(s)ds. $$
It is known that $e_n(I)\le c_0 n^{-1}$ (c.f. Theorem 2.1 of Lifshits and Linde \cite{LifshitsLinde2002}). It is obvious that 
$$\big\|S_{w_k}\big\|=\sup_{\|f\|_{[0,1]}\le 1}\|w_kf\|_{L^q[0,1]}=\|w_k\|_{L^q[0,1]}\le 1. $$
It follows that
$$ e_n(T_k)\le    e_n(I)\cdot   \big\|S_{w_k}\big\|\le c_0 n^{-1}. $$
By Lemma \ref{lem:3.8}, there exist constants $C$ and $\epsilon_0>0$ such that \eqref{eq:proof:lem:3.7:4} holds for all $0<\epsilon\le \epsilon_0$ and all $k$. 
On the other hand, by \eqref{eq:proof:lem:3.6:1},
\begin{align*}
-\log \pr\Big(   \big\|I\big(w_kX\big) \big\|_{L^q[0,1]}< \epsilon\Big)
\le  -\log \pr\Big(  \|X\|_{[0,1]}< \epsilon\Big)
\le C_{1/(\tau+1),\epsilon_0}\epsilon^{-1/(\tau+1)}, \;\; \epsilon>\epsilon_0.
\end{align*}
Thus, \eqref{eq:proof:lem:3.7:4} holds for all $\epsilon>0$ and all $k$.

 Next, we consider the second term of \eqref{eq:proof:lem:3.7:1}. Choose $\beta_0<1/(\tau+1)$.  For the centered normal random variable $\xi_k$, 
 $$ \sigma_k=(\Var(\xi_k))^{1/2}=\sqrt{\frac{\pi}{2}}\ep|\xi_k|\le C 2^{k/q}\int_0^{2^{k_0}}|w(s)|s^{\tau}ds=C_12^{k/q}. $$
 Let $\{\lambda_k\}$ be a sequence of positive numbers (for example $\lambda_k= 2^{k/2}/\sum_{j\le k_0}2^{j/2}$) such that $\sum_{k\le k_0}\lambda_k=1$ and 
 $$
 \sum_{k\le k_0}\big(\lambda_k^{-1/q} C_12^{k/q}\big)^{\beta_0}=:C_2<\infty.$$ 
 Then, by   the Gaussian correlation inequality \eqref{eq:GaussCor}, we have that
 \begin{align}\label{eq:proof:lem:3.7:5}
  -\log\pr\big(\eta_2<\epsilon\big)\le  & -\log\pr\Big(\bigcap_{k=-\infty}^{k_0}\{|\xi_k|<\epsilon \lambda_k^{1/q}\}\Big)
\le -\sum_{k\le k_0}\log\pr\Big( |\xi_k|<\epsilon \lambda_k^{1/q}\Big)\nonumber\\
=& -\sum_{k\le k_0}\log\pr\Big(\sigma_k |N(0,1)|<\epsilon \lambda_k^{1/q}\Big)
\le \sum_{k\le k_0}C_{\beta_0}\big(\epsilon^{-1} \lambda_k^{-1/q}\sigma_k\big)^{\beta_0}\nonumber\\
\le &C_{\beta_0}C_2 \epsilon^{-\beta_0} \;\text{ for all } \epsilon>0. 
 \end{align}
 By Lemma \ref{lem:ChungLIL1} and combining \eqref{eq:proof:lem:3.7:2} and \eqref{eq:proof:lem:3.7:5}, we have that
 \begin{align*}
 -\log\pr\left( \|I\big(wX\big)\|_{L^q(\Delta)}<\epsilon\right)\le -\log\pr\left(\eta_1<\epsilon/2,\eta_2<\epsilon/2\right)
\le C\epsilon^{-1/(\tau+1)} +o\big(\epsilon^{-1/(\tau+1)}\big).
 \end{align*}
 Thus \eqref{eq:lem:3.7:2} holds when $1< q<\infty$.

 (ii) For \eqref{eq:lem:3.7:3}, we define $Q_w$ by $Q_wf=I(wf)-wI(f)$.  
 Notice that $I(X)$ is a continuous, centered, self-similar Gaussian process of index $\tau+1$, and 
 $$ -\log \pr\left(\|I(X)\|_{[0,1]}<\epsilon\right)\preccurlyeq \epsilon^{-\tau/(\tau+1)}, $$
 by \eqref{eq:selfsimilar.5} of Lemma \ref{lem:selfsimilar}. Applying Lemma \ref{lem:3.6}, we have that
 $$
 -\log \pr\left(\|wI(X)\|_{L^q(0,\infty)}<\epsilon\right)\le C\|w\|_{r,\tau+1,q}^{1/(\tau+1)}\cdot \epsilon^{-1/(\tau+1)},\;\;  \epsilon> 0,
$$
which, together with \eqref{eq:lem:3.7:2}, implies that
\begin{align}\label{eq:proof:lem:3.7:6}
 -\log \pr\left( \left\|Q_wX\right\|_{L^q(\Delta)}<\epsilon\right) 
\le & -\log \pr\left( \left\|I\big(wX\big)\right\|_{L^q(\Delta)}<\epsilon/2, \left\|wI(X)\right\|_{L^q(\Delta)}<\epsilon/2 \right)\nonumber\\
\le & C\|w\|_{r,\tau+1,q}^{1/(\tau+1)}\cdot \epsilon^{-1/(\tau+1)}+o\big(\epsilon^{-1/(\tau+1)}\big),
\end{align}
by Lemma \ref{lem:WeakGaussCor} or \eqref{eq:GaussCor}.  When $q=\infty$, $\Delta$ can be chosen to be the whole interval $(0,\infty)$.  

For a function $w$ with $\|w\|_{r,\tau+1,q}<\infty$ and any given $\delta>0$, we can find a bounded and closed interval $\Delta\subset (0,\infty)$ such that $\|w-w\mathbbm{I}_{\Delta}\|_{r,\tau+1,q}<\delta/2$.  

 When $q<\infty$,  since $w\mathbbm{I}_{\Delta}\in L^q(I)$, we can find 
an interval step function $\widetilde{w}$  such that  $\|w\mathbbm{I}_{\Delta}-\widetilde{w}\|_{r,\tau+1,q}< \delta/2$.
It follows that  $\|w-\widetilde{w}\|_{r,\tau+1,q}<\delta$ and
\begin{align}\label{eq:proof:lem:3.7:7}  -\log \pr\left( \left\|Q_{w-\widetilde{w}}X\right\|_{L^q(0,\infty)}<\epsilon\right)\le C\delta^{1/(\tau+1)}\epsilon^{-1/(\tau+1)}
+o\big(\epsilon^{-1/(\tau+1)}\big), 
\end{align}
by \eqref{eq:proof:lem:3.7:6}. Thus,  for \eqref{eq:lem:3.7:3} it is sufficient to show that it holds for an interval step function $w$ of the form 
$ w=\sum_{j=1}^mw_j\mathbbm{I}_{(s_j,s_{j+1}]}. $
Then
$$\int_0^tw(s)X(s)ds-w(t)\int_0^tX(s)ds=\int_0^{s_j}\big(w(s)-w_j\big)X(s)ds=:\xi_j, \; t\in (s_j,s_{j+1}]. $$
Since $(\xi_1,\cdots,\xi_m)$ is a Gaussian vector, 
$$  \text{the left hand of \eqref{eq:lem:3.7:3}}\le -\log\pr\left(\max_{1\le j\le m}|\xi_j|<\epsilon\right)=o(\epsilon^{-\beta}), \; \text{ for all } \beta>0. 
$$

When $q=\infty$, it is also sufficient to show that \eqref{eq:lem:3.7:3} holds with  $w\mathbbm{I}_{\widetilde{\Delta}}$ taking the place of $w$, where $\widetilde{\Delta}$ is a bounded and closed sub-interval of $\Delta=[c,d]\subset (0,\infty)$. On $\widetilde{\Delta}$, $w$ is bounded and almost everywhere continuous. Also, $\|Q_wX\|_{L^{\infty}(0,\infty)}=
\|Q_wX\|_{L^{\infty}(\Delta)}\vee |\int_0^d w(s)X(s)ds|$.  Thus, without loss of generality, we can assume that $\Delta\subset [0,1]$ and  $w$ is bounded and almost everywhere continuous on $\Delta$. Then, $w\in L^p(\Delta)$ for all $p>0$.  Denote $T_{\rho,\psi}$ by
$T_{\rho,\psi}f=\rho I(\psi f)$. Then $Q_w=T_{1,w}-T_{w,1}$. By Theorem 4.6 (2) of Lifshits and Linde \cite{LifshitsLinde2002}, 
$$ \limsup_{n\to\infty} n \cdot e_n\big(Q_w:L^{\infty}(\Delta)\to L^{\infty}(\Delta)\big)\le c\cdot (\|w\|_{L^1(\Delta)}+\|w^{\ast}\|_{L^1(\Delta)})= 2c\cdot  \|w\|_{L^1(\Delta)}, $$
since $|w|=w^{\ast}$ almost everywhere by noting the almost everywhere continuity of $w$, where $w^{\ast}$ is defined as \eqref{eq:th:3.1:2}, $c$ is a constant which does not depend on $w$. \red{Write $\overline{w}=w/\|w\|_{L^1(\Delta)}$. Then 
$$ \limsup_{n\to\infty} n \cdot e_n\big(Q_{\overline{w}}:L^{\infty}(\Delta)\to L^{\infty}(\Delta)\big)\le  c. $$ 
By Lemma \ref{lem:3.8}, there exists a universal constant $C_{\Delta}$  
(which may depend on $X$, $\Delta$)  such that
\begin{align*}
 -&\log \pr\left(\|Q_wX\|_{L^{\infty}(\Delta)}<\epsilon\right)
=-\log\pr\left(\|Q_{\overline{w}}X\|_{L^{\infty}(\Delta)}<\epsilon/\|w\|_{L^1(\Delta)}\right) \\
& \;\; \le   C_{\Delta} \cdot\left(\epsilon/\|w\|_{L^1(\Delta)}\right)^{-1/(\tau+1)}
+o\left(\epsilon/\|w\|_{L^1(\Delta)}\right)^{-1/(\tau+1)}\\
& \;\; =   C_{\Delta} \cdot\|w\|_{L^1(\Delta)}^{1/(\tau+1)}\cdot\epsilon^{-1/(\tau+1)}+o\big(\epsilon^{-1/(\tau+1)}\big). \end{align*}
}For this function $w$  and any given $\delta>0$, since $w\in L^1(\Delta)$, there exists an interval step function $\widetilde{w}$ on $\Delta$ such that 
$\|w-\widetilde{w}\|_{L^1(\Delta)}<\delta$ and $\|w-\widetilde{w}\|_{L^{\infty}(\Delta)}\le 2\|w\|_{L^{\infty}(\Delta)}$.  Thus, \eqref{eq:proof:lem:3.7:7} remains true.
The proof is now completed. $\Box$

{\bf Proof of Theorem \ref{th:3.2}}. Recall $\tau=H+m+\gamma$. Let
$$X(t)= a_H^{-1} t^{\alpha_1+\cdots+\alpha_m}   J_{m,\bm\alpha}(I_{\gamma}B_H)(t) \;\; \text{ or }t^{\alpha_1+\cdots+\alpha_m}   J_{m,\bm\alpha}(W_{H+\gamma})(t). $$
Then $X(t)$ is a continuous, centered, self-similar Gaussian process of index $\tau$, and 
$$-\log\pr\left(\|X\|_{[0,1]}<\epsilon\right)\preccurlyeq \epsilon^{-1/\tau}, $$
by \eqref{eq:th:1:1:1}. 

We first assume that $w$ is almost everywhere continuous when $q=\infty$. Then, by \eqref{eq:lem:3.7:3} of Proposition \ref{lem:3.7},   
$$ -\log\pr\left(\|I(wX)-wI(X)\|_{L^q(\Delta)}<\epsilon\right)=o(\epsilon^{-1/(\tau+1)}). $$
On the other hand,
by \eqref{eq:proof:th:3.1:0}, 
$$-\log \pr\left(\|X-W_{\tau}\|_{[0,1]}< \epsilon\right)\preccurlyeq \epsilon^{-1/(\tau+1)}.$$  
Thus, $I(X)-W_{\tau+1}=I(X-W_{\tau})$ is a continuous, centered, self-similar Gaussian process of index $\tau+1$, and
$$-\log \pr\left(\|I(X)-W_{\tau+1}\|_{[0,1]}< \epsilon\right)\preccurlyeq \epsilon^{-1/(\tau+2)}=o(\epsilon^{-1/(\tau+1)}),$$ 
 by Lemma \ref{lem:selfsimilar} (ii). Applying Lemma \ref{lem:3.6}, we have that
 $$-\log \pr\left(\|wI(X)-wW_{\tau+1}\|_{L^q(0,\infty)}< \epsilon\right)=o(\epsilon^{-1/(\tau+1)}).$$   
 It follows that
$$ -\log \pr\left(\|I(wX)-wW_{\tau+1}\|_{L^q(\Delta)}< \epsilon\right)=o(\epsilon^{-1/(\tau+1)}).
$$
 Then, \eqref{eq:th:3.2:1} follows from \eqref{eq:proof:th:3.1:1} with $\lambda=\tau+1$.  
 
 \red{Next, we remove the condition that $w$ is almost everywhere continuous when $q=\infty$. 
 Since there is a closed and bounded interval $\widetilde{\Delta}\subset \Delta$ such that 
 $\|w\mathbbm{I}_{\Delta}-\widetilde{w}\|_{r,\tau+1,q}<\delta$, and
 \begin{align}  -\log \pr\left( \left\|I\big((w-\widetilde{w})X\big)\right\|_{L^q(\Delta)}<\epsilon\right)\le C\delta^{1/(\tau+1)}\epsilon^{-1/(\tau+1)}
+o\big(\epsilon^{-1/(\tau+1)}\big), 
\end{align}
 by \eqref{eq:lem:3.7:2}. Thus again, without loss of generality we can assume that $w$ has a support on a closed and bounded  subinterval $\widetilde{\Delta}$ of $\Delta$. 
 By Theorem 4.6 (2) of Lifshits and Linde \cite{LifshitsLinde2002}, 
$$ \limsup_{n\to\infty} n \cdot e_n\big(T_{1,w}:L^{\infty}(\widetilde{\Delta})\to L^{\infty}(\widetilde{\Delta})\big)\le c\cdot  \|w\|_{L^1(\widetilde{\Delta})}, $$
which, together with  Lemma \ref{lem:3.8}, implies that
\begin{align*}
 -&\log \pr\left(\|I(wX)\|_{L^{\infty}(\Delta)}<\epsilon\right)
= \log \pr\left(\|I(wX)\|_{L^{\infty}(\widetilde{\Delta})}<\epsilon\right)\\
& \;\;\le   C_{\widetilde{\Delta}} \cdot\|w\|_{L^1(\widetilde{\Delta})}^{1/(\tau+1)}\cdot\epsilon^{-1/(\tau+1)}+o\big(\epsilon^{-1/(\tau+1)}\big). \end{align*}
where $C_{\widetilde{\Delta}}$ does not depend on $w$. For this $w$, there exists an interval step function $\widetilde{w}$ such that $\|w-\widetilde{w}\|_{L^1(\widetilde{\Delta})}<\delta$.
Thus,
 \begin{align*}
 -&\log \pr\left(\|I\big((w-\widetilde{w})X\big)\|_{L^{\infty}(\Delta)}<\epsilon\right)\\
 \le &  C_{\widetilde{\Delta}} \delta^{1/(\tau+1)}\cdot\epsilon^{-1/(\tau+1)}+o\big(\epsilon^{-1/(\tau+1)}\big). \end{align*}
 For the interval step function $\widetilde{w}$, \eqref{eq:th:3.2:1} holds since $\widetilde{w}$ is almost everywhere continuous. Thus, \eqref{eq:th:3.2:1} holds for all $w$ under the condition \eqref{eq:th:3.2:cond}. 
 }
 Finally, notice $\|I(wX)\|_{L^{\infty}(\Delta)}=\|I(wX)\|_{\Delta}$ since $I(wX)$ is continuous. The proof is completed. $\Box$
 
\section{An application to randomized play-the-winner rule}
Consider an urn with two types of balls (white and black) which starts at $W_0> 0$ white balls and $B_0>0$ black balls. At each stage, we a draw ball from the urn with replacement.   If a white ball is drawn, then an additional white ball or black ball is added to the urn with a probability $p_W$ and $q_W=1-p_W$, respectively. 
If a black ball is drawn, then an additional black ball or white ball is added to the urn with a probability $p_B$ and $q_B=1-p_B$, respectively. This urn model is the randomized-play-the-winner (RPW) rule  introduced by Wei and Durham \cite{WD78}  for   sequentially randomizing patients to treatments in a clinical trial.  After $n$ generations,   the number of white balls in the urn is denoted by $Y_n$, and, the number of white balls drawn is denoted by $N_n$. 
 Let $\rho=p_W+p_B-1$, $v=q_B/(q_W+q_B)$, $\sigma_1^2
=q_Wq_B/(q_W+q_B)^2$, $\sigma_2^2=q_Wq_B(p_W+p_B)/(q_W+q_B)$. Then $\rho^2\sigma_1^2+\sigma_2^2=\sigma_1^2 $.  Suppose $0<p_W,p_B<1$  and $\rho<1/2$. Bai, Hu and Zhang \cite{BHZ2002} and  Zhang and Hu \cite{ZH2009} showed the  Gaussian approximation of $Y_n$ and $N_n$ as that
$$ Y_n-n v= G_1(n)+o(n^{1/2-\gamma}) \;a.s. \text{ and}, $$  
$$ N_n-n v= G_2(n) +o(n^{1/2-\gamma}) \;a.s., $$
where $\gamma>0$, 
\begin{align*}
G_1(t)=& \rho\sigma_1B_1(t)+\sigma_2B_2(t)+  \rho t^{\rho}\int_0^t \frac{\rho\sigma_1B_1(s)+\sigma_2B_2(s)}{s^{\rho+1}}ds,\\
 G_2(t)=& \sigma_1B_1(t) +t^{\rho}\int_0^t \frac{\rho\sigma_1B_1(s)+\sigma_2B_2(s)}{s^{\rho+1}}ds,
 \end{align*}
$B_1(t)$ and $B_2(t)$ are two independent standard Brownian motions (c.f. Theorem 4.4 of Zhang and Hu \cite{ZH2009}). By the Gaussian approximation, Bai, Hu and Zhang \cite{BHZ2002} and Zhang \cite{Zhang2012} obtained the following law of the iterated logarithm:
\begin{align*}
& \limsup_{n\to \infty} \frac{|Y_n-nv|}{\sqrt{2n\log\log n}}=\frac{\sigma_1}{\sqrt{1-2\rho}} \; a.s. \text{ and } \\
& \limsup_{n\to \infty} \frac{|N_n-nv|}{\sqrt{2n\log\log n}}=\sqrt{\frac{\sigma_1^2(1+2(p_W+p_B))}{1-2\rho}} \; a.s.
\end{align*}
(c.f. Theorem 4.3 of Zhang \cite{Zhang2012}). 

Now, notice that $\{\rho\sigma_1B_1(t)+\sigma_2B_2(t), t\ge 0\}\overset{d}= \{\sigma_1B(t), t\ge 0\}$.
  By  \eqref{eq:cor:1:2}, 
$$\lim_{\epsilon\to 0}\epsilon^{2/3} \log \pr\left(\sup_{0\le t\le 1}t^{\rho}\left|\int_0^t \frac{\rho\sigma_1B_1(s)+\sigma_2B_2(s)}{s^{\rho+1}}ds\right|<\epsilon\right) 
= - 3\kappa_{\frac{3}{2}}\sigma_1^2. $$
 By Lemma \ref{lem:ChungLIL1}, it follows that
\begin{align*}
&\lim_{\epsilon\to 0}\epsilon^2\log \pr\left(\sup_{0\le t\le 1}\left|G_i(t)\right|<\epsilon\right)=\lim_{\epsilon\to 0}\epsilon^2 \log \pr\left(\sigma_1\sup_{0\le t\le 1}\left|B(t)\right|<\epsilon\right)=-\frac{\sigma_1^2\pi^2}{8},\; i=1,2,
\end{align*}
which implies that 
$$\liminf_{T\to \infty} \sqrt{ \frac{\log \log T}{T}}\sup_{0\le t\le T}\big|G_i(t)\big|= \frac{\sigma_1 \pi}{\sqrt{8}}\; a.s., \;\; i=1,2. $$
Hence, we have the following Chung-tye law of the iterated logarithm:
\begin{align*}
   \liminf_{n\to \infty}\sqrt{ \frac{\log \log n}{n}}\sup_{1\le m\le n}\big|Y_m-mv\big|= \liminf_{n\to \infty} \sqrt{ \frac{\log \log n}{n}}\sup_{1\le m\le n}\big|N_m-mv\big|= \frac{\sigma_1 \pi}{\sqrt{8}}\; a.s.
\end{align*}


\end{document}